\documentclass[11pt,reqno]{amsart}

\usepackage{amssymb}
\usepackage{mathdots} %
\usepackage{bbm} %
\usepackage[shortlabels]{enumitem} %
\usepackage{geometry}
\usepackage{adjustbox} %
\usepackage{verbatim} %

\usepackage{hyperref}
\usepackage{xcolor}
\definecolor{darkgreen}{rgb}{0,0.45,0}
\hypersetup{colorlinks,urlcolor=blue,citecolor=darkgreen,linkcolor=blue,linktocpage=false}
\usepackage[abbrev]{amsrefs}
\usepackage[all,color]{xy}
\newdir{ >}{{}*!/-8pt/@{>}} %
\SilentMatrices
\usepackage{float} %
\usepackage{pgfplots}
\pgfplotsset{compat=newest}

\usepackage{tikz}
\usetikzlibrary{cd}
\usetikzlibrary{decorations.markings}

\tikzset{
 ->-/.style = {
  decoration = {markings, mark=at position #1 with {\arrow[scale=1.3]{>}}},
  postaction = {decorate}
 }
}

\usepackage{ifpdf}
\ifpdf
\else
\usepackage{breakurl}
\fi

\usepackage{mathtools} %

\numberwithin{equation}{section}

\theoremstyle{plain}
\newtheorem{thm}[equation]{Theorem}
\newtheorem*{thm*}{Theorem} %
\newtheorem{prop}[equation]{Proposition}
\newtheorem{lem}[equation]{Lemma}
\newtheorem{cor}[equation]{Corollary}
\newtheorem*{cor*}{Corollary} %

\newtheorem{innercustomthm}{Theorem}
\newenvironment{customthm}[1]
  {\renewcommand\theinnercustomthm{#1}\innercustomthm}
  {\endinnercustomthm}

\theoremstyle{definition}
\newtheorem{defn}[equation]{Definition}
\newtheorem{ex}[equation]{Example}
\newtheorem{nota}[equation]{Notation}

\theoremstyle{remark}

\newtheorem{rem}[equation]{Remark}

\makeatletter
\newcommand{\Eqref}[1]{\textup{\tagform@{\ref*{#1}}}}
\makeatother

\newcommand{\anti}{\ar@{}[dr]|{\;\;\;\boxed{\textup{\scriptsize -1}}}} %
\newcommand{\pb}{\ar@{}[dr]|{\mbox{\huge $\lrcorner$}}} %
\newcommand{\po}{\ar@{}[dr]|{\mbox{\huge $\ulcorner$}}} %

\newcommand{\Def}[1]{\textbf{\boldmath{#1}}} %
\newcommand{\dfn}{:=}

\newcommand{\C}{\mathbb{C}}

\newcommand{\N}{\mathbb{N}}
\newcommand{\kk}{\Bbbk} %
\newcommand{\PPP}{\mathbb{P}}

\newcommand{\R}{\mathbb{R}}
\newcommand{\Z}{\mathbb{Z}}

\newcommand{\DD}[1]{\mathcal{L}(#1)} %
\newcommand{\PP}[1]{\mathcal{P}(#1)} %

\newcommand{\cat}[1]{\mathbf{\mathcal{#1}}} %

\newcommand{\bmat}[1]{\left[\!\!\!\begin{array}{#1}} %
\newcommand{\emat}{\end{array}\!\!\right]} %

\newcommand{\smat}[1]{\left[\begin{smallmatrix*}[r]#1\end{smallmatrix*}\right]} %

\newcommand{\al}{\alpha}
\newcommand{\be}{\beta}
\newcommand{\de}{\delta}
\newcommand{\De}{\Delta}
\newcommand{\del}{\partial}
\newcommand{\ep}{\epsilon}
\newcommand{\ga}{\gamma}

\newcommand{\io}{\iota}

\newcommand{\la}{\lambda}

\newcommand{\phy}{\varphi}

\newcommand{\si}{\sigma}
\newcommand{\Te}{\Theta}

\newcommand{\inj}{\hookrightarrow}
\newcommand{\ral}{\xrightarrow} %
\newcommand{\Ra}{\Rightarrow} %
\newcommand{\surj}{\twoheadrightarrow}

\newcommand{\op}{\oplus}
\newcommand{\ot}{\otimes}
\newcommand{\x}{\times}

\newcommand{\lan}{\left\langle}
\newcommand{\ol}{\overline}
\newcommand{\pvec}[1]{\vec{#1}\mkern2mu\vphantom{#1}} %
\newcommand{\ran}{\right\rangle}
\newcommand{\tild}{\widetilde}

\newcommand{\Mod}[1]{#1\text{-}\mathrm{Mod}}
\newcommand{\fgMod}[1]{#1\text{-}\mathrm{mod}} %
\newcommand{\fdMod}[1]{#1\text{-}\mathrm{fmod}} %
\newcommand{\Topp}{\mathrm{Top}}
\newcommand{\tors}[1]{#1\text{-}\mathrm{tors}}
\newcommand{\vect}[1]{\mathrm{vect}_{#1}}
\newcommand{\Vect}[1]{\mathrm{Vect}_{#1}}

\newcommand{\abs}[1]{\lvert #1 \rvert}
\DeclarePairedDelimiter{\ceil}{\lceil}{\rceil} %

\DeclareMathOperator{\Ass}{Ass}
\DeclareMathOperator{\coker}{coker}
\DeclareMathOperator*{\colim}{colim}
\DeclareMathOperator{\Coh}{Coh}

\DeclareMathOperator{\Fun}{Fun} %
\DeclareMathOperator{\Hom}{Hom}
\DeclareMathOperator{\im}{im}
\DeclareMathOperator{\Kdim}{Kdim}
\DeclareMathOperator{\Obj}{Obj}
\DeclareMathOperator{\rank}{rank}
\DeclareMathOperator{\Rep}{Rep}
\DeclareMathOperator{\rep}{rep}
\DeclareMathOperator{\rk}{rk} %
\DeclareMathOperator{\irk}{irk} %
\DeclareMathOperator{\sirk}{sirk} %
\DeclareMathOperator{\sk}{sk}
\DeclareMathOperator{\Spec}{Spec}
\newcommand{\SR}[1]{\kk[#1]} %
\DeclareMathOperator{\Supp}{Supp}

\newcommand{\cst}{\mathrm{cst}}
\newcommand{\id}{\mathrm{id}}

\newcommand{\inc}{\mathrm{inc}}

\def\noteson{\gdef\note##1{\noindent{\color{blue}[##1]}}}

\noteson

\newcommand{\VertexSize}{0.05}

\begin{document}

\title{Multiparameter persistence modules in the large scale} 
\date{\today}

\author{Martin Frankland} 
\address{University of Regina\\
3737 Wascana Parkway\\
Regina, Saskatchewan, S4S 0A2\\
Canada}
\email{Martin.Frankland@uregina.ca}

\author{Donald Stanley}
\address{University of Regina\\
3737 Wascana Parkway\\
Regina, Saskatchewan, S4S 0A2\\
Canada}
\email{Donald.Stanley@uregina.ca}

\begin{abstract}
A persistence module with $m$ discrete parameters is a diagram of vector spaces indexed by the poset $\N^m$. If we are only interested in the large scale behavior of such a diagram, then we can consider two diagrams equivalent if they agree outside of a ``negligeable'' region. In the $2$-dimensional case, we classify the indecomposable diagrams up to finitely supported diagrams. In higher dimension, we partially classify the indecomposable diagrams up to suitably finite diagrams, and show that the full classification problem is wild.
\end{abstract}

\keywords{persistence module, persistent homology, multiparameter, indecomposable, abelian category, Krull--Schmidt} %

\subjclass[2020]{Primary 18E10; Secondary 55N31, 18E35}

\maketitle

\tableofcontents

\section{Introduction}

In topological data analysis, it is common to study a space endowed with a filtration, which corresponds to a functor $\N \to \Topp$ in the case of a discrete parameter, or $\R_{\geq 0} \to \Topp$ in the case of a continuous real parameter. %
A multiparameter filtration using $m \geq 1$ discrete parameters corresponds to a functor $\N^m \to \Topp$ from the poset $\N^m$. Taking the $i\textsuperscript{th}$ homology group with coefficients in a field $\kk$ then yields a functor $\N^m \to \Vect{\kk}$ to $\kk$-vector spaces, called %
an \emph{$m$-parameter persistence module}. These correspond to multigraded modules over the $\N^m$-graded ring $R = \kk[t_1, \ldots, t_m]$ with grading $\abs{t_i} = \vec{e}_i = (0, \ldots, 1, \ldots, 0)$. 
For background on multiparameter persistence modules and their applications to topological data analysis, see \cite{CarlssonZ09}, \cite{Oudot15}, \cite{Lesnick15}, as well as the comprehensive survey \cite{BotnanL23}. 
This paper concerns the representation theory of finitely generated $R$-modules. 

Each finitely generated $R$-module decomposes (uniquely) into a direct sum of indecomposable modules; 
see \cite{BotnanC20} for more general decomposition theorems. 
For one-parameter persistence modules (the case $m=1$), by the classification of finitely generated modules over a principal ideal domain, each module decomposes into a direct sum of interval modules, which are the indecomposables. The intervals appearing in the decomposition can then be summarized into a barcode. 
For multiparameter persistence modules (the case $m \geq 2$), such a classification of indecomposables is unavailable, since the graded ring $\kk[t_1,t_2]$ has wild representation type; see for instance \cite{BotnanL23}*{\S 8}. 
One approach around that problem is to extract from a module certain invariants, notably the rank invariant (introduced by Carlsson and Zomorodian \cite{CarlssonZ09}) and its various refinements. The goal is to find invariants that are both computable and significant in applications. That approach is developed for instance in \cite{Vipond20}, \cite{GafvertC21}, \cite{ChacholskiCS24}, \cite{BotnanOO22}, \cite{OudotS24}, \cite{KimM24}, \cite{BlanchetteBH22}, and \cite{ClauseDMW23}.
Another approach is to focus on certain families of modules admitting a nice decomposition, such as rectangle-decomposable modules. 
One then tries to characterize the modules that are of the special form, and to approximate an arbitrary module by such a module. %
See for instance \cite{BotnanLO22}, %
\cite{BotnanLO23}, %
and \cite{AsashibaBENY22}.

In this paper, we take a different approach. We localize the category of $R$-modules until the resulting category admits a classification of indecomposables, or at least a partial classification.

\subsection*{Organization and results}

Given a simplicial complex $K$ on $m$ vertices, we define an abelian category $\DD{K}$ of \emph{$K$-localized persistence modules} (Definition~\ref{def:LocalModule}). Those encode the data of various localizations of an $R$-module, where $K$ parametrizes which variables among $t_1, \ldots, t_m$ have been inverted. We exhibit $\DD{K}$ as a Serre quotient of $R$-modules (Proposition~\ref{pr:SerreQuotient}).

Sections~\ref{sec:TorsionTh}--\ref{sec:Torsionfree} then focus on the simplicial complex $K_m \coloneqq \sk_{m-3} \De^{m-1}$, the $(m-3)$-skeleton of the standard simplex $\De^{m-1}$. One reason for that choice is that for any smaller simplicial complex $K' \subset K_m$, the category $\DD{K'}$ is guaranteed to have wild representation type (Proposition~\ref{pr:DecompoTorsion}). We construct a torsion pair on $\DD{K_m}$ %
and show that the torsion pair splits, so that every object is a direct sum of its torsion part and its torsion-free part (Proposition~\ref{pr:allareTT}). %
We further decompose the torsion objects into direct sums of indecomposable objects of a specific form (Proposition~\ref{pr:DecompoTorsion}). In the case $m=2$, we decompose the torsion-free objects into direct sums of indecomposable objects of a specific form (Proposition~\ref{pr:M2Decompo}). For $m \geq 3$, we show that the analogous decomposition of torsion-free objects does \emph{not} hold (Proposition~\ref{pr:Mgeq3}). The results for $\DD{K_2}$ can be summarized as follows.

\begin{customthm}{A}\label{thm:ThmA}
In the category of %
finitely generated $\kk[s,t]$-modules 
modulo the subcategory of finite modules, 
every object decomposes in a unique way as a direct sum of:
\begin{itemize}
\item ``vertical strips'' $\displaystyle [a,b)_1 \dfn s^{a} \kk[s,t] / s^{b} \kk[s,t]$
\item ``horizontal strips'' $\displaystyle [a,b)_2 \dfn t^{a} \kk[s,t] / t^{b} \kk[s,t]$
\item ``quadrants'' $\displaystyle [(a_1,a_2),\infty) \dfn s^{a_1} t^{a_2} \kk[s,t]$.
\end{itemize}
See Figure~\ref{fig:Strips}.
\end{customthm}

In particular, $\DD{K_2}$ does not have wild representation type. Combining this fact with Proposition~\ref{pr:WildRetract} and Theorem~\ref{thm:WildRep}, we obtain:

\begin{customthm}{B}\label{thm:ThmB}
The category $\DD{K}$ has wild representation type if and only if the simplicial complex~$K$ has more than $3$ missing faces. In other words: a missing face of codimension $2$ or at least $3$ missing faces of codimension $1$.
\end{customthm}

In Section~\ref{sec:Rank}, we show that the rank invariant of an $R$-module $M$ determines the torsion part of its $K_m$-localization $L_{K_m}(M)$ in $\DD{K_m}$, but not the torsion-free part. %
In the case $m=2$, we find a refinement of the rank invariant that does determine the torsion-free part (Theorem~\ref{thm:RefinedRank}). 
In Section~\ref{sec:Classif}, we turn our attention to the Serre subcategory of $R$-modules being quotiented out in the construction of $\DD{K}$. We classify all %
tensor ideals 
of $\fgMod{R}$ and show that they are in bijection with simplicial complexes $K$ on $m$ vertices (Theorem~\ref{thm:Classif}). 
The %
tensor ideal 
corresponding to $K$ is generated by the Stanley--Reisner ring $\SR{K}$. In Section~\ref{sec:Simple}, we revisit $\DD{K_m}$ and show that it is obtained from $\fgMod{R}$ by iteratively quotienting out the simple objects $m-1$ times (Corollary~\ref{cor:ModOutSimples}).

\begin{rem}
The categories we consider arise in algebraic geometry. The category of coherent sheaves on the projective line $\PPP^1$ over the base field $\kk$ is described by an equivalence
\[
\Coh(\PPP^1) \cong \Z\text{-graded } \fgMod{\kk[s,t]} / \{ \text{finite modules} \},
\]
where $\kk[s,t]$ is given an $\N$-grading with $\abs{s} = \abs{t} = 1$, and the right-hand side denotes the Serre quotient by the subcategory of graded modules $M$ with graded pieces $M(d) \neq 0$ in finitely many degrees $d \in \Z$. See \cite{Hartshorne77}*{Exercise~II.5.9} and \cite{stacks-project}*{Tag~0BXD}, as well as \cite{Cox95}*{\S 3} and \cite{CoxLS11}*{\S 5.3}. 
The $\Z^2$-graded variant of 
our category $\DD{K_2}$ is the bigraded analogue of the right-hand side:
\[
\DD{K_2}_{\Z^2} \cong \Z^2\text{-graded } \fgMod{\kk[s,t]} / \{ \text{finite modules} \}.
\]
It was kindly pointed out to us by Colin Ingalls that the latter category is equivalent to \emph{torus-equivariant} coherent sheaves on $\PPP^1$; see for instance \cite{Perling04}*{\S 5} and \cite{Perling13}. 
Hence Theorem~\ref{thm:ThmA} provides a decomposition result for such sheaves. It would be interesting to pursue the connections between persistence modules and toric geometry.
\end{rem}

\subsection*{Related work}

This paper is close in spirit to \cite{HarringtonOST19} and uses some similar tools. In that paper, the authors extract computable %
invariants from $R$-modules that are related to the localizations we consider, whereas our focus is the decomposition of objects in the localized categories.

The paper \cite{BauerBOS20} is also closely related. The authors use torsion pairs and consider decompositions of certain families of $2$-parameter persistence modules up to certain subcategories.

\subsection*{Funding}

This work was supported by the Natural Sciences and Engineering Research Council of Canada [RGPIN-2019-06082 to M.F., RGPIN-05466-2020 to D.S.].

\subsection*{Acknowledgments}

We thank Colin Ingalls, Thomas Br\"ustle, and Markus Perling for helpful discussions. 
We are very grateful to Steffen Oppermann for providing the argument in the proof of Theorem~\ref{thm:WildRep}. 
We also thank anonymous referees for their comments. 
Frankland acknowledges the support of the Max-Planck-Institut für Mathematik, report number MPIM-Bonn-2026. 

\section{Serre quotients}

In this section, we collect a few facts about Serre quotients that are found in \cite{stacks-project}*{Tag~02MN}, \cite{GelfandM03}*{\S II.5 Exercise~9}, \cite{Popescu73}*{\S 4.3, 4.4}, or \cite{Gabriel62}*{\S III}.

\begin{defn}
Let $\cat{A}$ be an abelian category. A \Def{Serre subcategory} of $\cat{A}$ is a full subcategory $\cat{S} \subseteq \cat{A}$ that contains $0$ %
and such that for every short exact sequence
\[
\xymatrix{
0 \ar[r] & A \ar[r] & B \ar[r] & C \ar[r] & 0 \\
}
\]
in $\cat{A}$, the object $B$ lies in $\cat{S}$ if and only if $A$ and $C$ lie in $\cat{S}$. In other words, $\cat{S}$ is closed under forming subobjects, quotients, and extensions.
\end{defn}

\begin{lem}
Let $F \colon \cat{A} \to \cat{B}$ be an exact functor between abelian categories. Then the subcategory
\[
\ker(F) \dfn \{ A \in \cat{A} \mid F(A) \cong 0 \}
\]
is a Serre subcategory of $\cat{A}$, called the \Def{kernel} of $F$.
\end{lem}

\begin{lem}
Let $\cat{S} \subseteq \cat{A}$ be a Serre subcategory. Then there exists an abelian category $\cat{A}/\cat{S}$ and an exact functor $q \colon \cat{A} \to \cat{A}/\cat{S}$ satisfying the following universal property: For every exact functor $F \colon \cat{A} \to \cat{B}$ satisfying $\cat{S} \subseteq \ker(F)$, there exists a unique exact functor $\ol{F} \colon \cat{A}/\cat{S} \to \cat{B}$ satisfying $F = \ol{F} \circ q$, as illustrated in the diagram
\begin{equation}\label{eq:SerreQuotient}
\xymatrix{
\cat{A} \ar[d]_{q} \ar[r]^-{F} & \cat{B}. \\
\cat{A}/\cat{S} \ar@{-->}[ur]_{\ol{F}} & \\
}
\end{equation}
The category $\cat{A}/\cat{S}$ is called the \Def{Serre quotient} of $\cat{A}$ by %
$\cat{S}$.
\end{lem}

\begin{lem}[\cite{Popescu73}*{Lemma~4.3.7}]\label{lem:SEquiv}
Let $\cat{S} \subseteq \cat{A}$ be a Serre subcategory. A map $f \colon X \to Y$ in $\cat{A}$ becomes an isomorphism $q(f) \colon q(X) \ral{\cong} q(Y)$ in the Serre quotient $\cat{A}/\cat{S}$ if and only if it satisfies $\ker(f) \in \cat{S}$ and $\coker(f) \in \cat{S}$. Call such a map an \Def{$\cat{S}$-equivalence}.
\end{lem}

\begin{lem}\label{lem:SerreFaithful}
In the setup of diagram~\eqref{eq:SerreQuotient}, the induced functor $\ol{F} \colon \cat{A}/\cat{S} \to \cat{B}$ is faithful if and only if $\cat{S} = \ker(F)$ holds.
\end{lem}

The following statement is found in \cite{Gabriel62}*{Proposition~III.2.5} and \cite{Popescu73}*{Theorem 4.4.9}.

\begin{lem}\label{lem:SerreEquiv}
 Let $F \colon \cat{A} \to \cat{B}$ be an exact functor between abelian categories. If $F$ admits a right adjoint $G \colon \cat{B} \to \cat{A}$ that is fully faithful (equivalently, such that the counit $\ep \colon FG \Ra 1$ is an isomorphism), then the induced functor $\ol{F} \colon \cat{A}/ \ker(F) \to \cat{B}$ is an equivalence of categories.
\end{lem}

\section{Localization of persistence modules}

Let us fix some notation that will be used throughout the paper. 

\begin{nota}
We work over a base field $\kk$. 
Let $\Vect{\kk}$ denote the category of $\kk$-vector spaces and 
$\vect{\kk}$ the category of finite-dimensional $\kk$-vector spaces.
\end{nota}

\begin{nota}
Let $m \geq 1$ be an integer and denote the set $[m] \dfn \{ 1, 2, \ldots, m \}$. 
Consider the polynomial ring $R \dfn \kk[t_1, \ldots, t_m]$, which is $\N^m$-graded with multigrading
\[
\abs{t_i} = \vec{e}_i = (0, \ldots, \overbrace{1}^{i}, \ldots, 0).
\]
In other words, $R$ is the monoid algebra $R = \kk (\N^m)$. 
Let $\Mod{R}$ denote the category of $\N^m$-graded (left) $R$-modules and $\fgMod{R}$ the full subcategory of finitely generated $R$-modules. 
All $R$-modules will be $\N^m$-graded unless otherwise noted.
\end{nota}

\begin{nota}
Given a small category $I$ and a category $\cat{C}$, an \Def{$I$-shaped diagram} in $\cat{C}$ is a functor from $I$ to $\cat{C}$. The category of all such diagrams is denoted $\cat{C}^I = \Fun(I,\cat{C})$. 
\end{nota}

Viewing the poset $\N^m$ as a category, an $\N^m$-shaped diagram of $\kk$-vector spaces $M \colon \N^m \to \Vect{\kk}$ corresponds to a (multigraded) $R$-module $M$, called a \Def{persistence module}. This correspondence forms an isomorphism %
of categories
\[
\Fun(\N^m,\Vect{\kk}) \cong \Mod{R}
\]
\cite{HarringtonOST19}*{Theorem~2.6} \cite{CarlssonZ09}. 
Given an $R$-module $M$ and $\vec{d} = (d_1, \ldots, d_m) \in \N^m$, let $M(\vec{d})$ denote the $\kk$-vector space which is the part of $M$ in multidegree $\vec{d}$.

\begin{nota}
For a subset $\si \subseteq [m]$, denote the localization of monoids $\si^{-1} \N^m$, where we identify $i \in [m]$ with the standard basis element $\vec{e}_i \in \N^m$, for instance:
\[
\{ 1,3 \}^{-1} \N^3 = \Z \x \N \x \Z.
\]
Denote the localization of rings
\[
R_{\si} \dfn R[t_i^{-1} \mid i \in \si],
\]
which is $\si^{-1} \N^m$-graded. %
For instance, $R_{[m]} = \kk[t_1^{\pm}, \ldots, t_m^{\pm}]$ is $\Z^m$-graded. 
Likewise, $R_{\si}$-modules are $\si^{-1} \N^m$-graded. 
An $R$-module $M$ can be viewed as $\si^{-1} \N^m$-graded by setting
\[
M(\vec{d}) = 0 \quad \text{if } \vec{d} \in \si^{-1} \N^m \setminus \N^m.
\]
Denote by $\Mod{R}_{\si^{-1} \N^m}$ the category of $\si^{-1} \N^m$-graded $R$-modules, which can be viewed as a full subcategory of $\Z^m$-graded $R$-modules, denoted $\Mod{R}_{\Z^m}$. In this notation, we have $\Mod{R} = \Mod{R}_{\N^m}$ and $\Mod{R_{\si}} = \Mod{R_{\si}}_{\si^{-1} \N^m}$.
\end{nota}

The purpose of the category $\Mod{R}_{\si^{-1} \N^m}$ is to allow negative shifts of (multigraded) $R$-modules in certain directions.

\begin{ex}
In the case $m=2$ with $R = \kk[t_1,t_2]$, the $R$-module $M = t_1^{-5} R$ is an object of $\Mod{R}_{\{1\}^{-1} \N^2} = \Mod{R}_{\Z \x \N}$, but is \emph{not} an object of $\Mod{R}_{\{2\}^{-1} \N^2} = \Mod{R}_{\N \x \Z}$. Note that $M$ is \emph{not} a module over the localization $R_{\{1\}} = \kk[t_1^{\pm}, t_2]$, since $t_1 \in R$ does not act invertibly on $M$. 
\end{ex}

\begin{lem}\label{lem:ZmGraded}
The full subcategory $\Mod{R}_{\si^{-1} \N^m} \subseteq \Mod{R}_{\Z^m}$ is an abelian subcategory.
\end{lem}

\begin{proof}
Limits and colimits in $\Mod{R}_{\si^{-1} \N^m}$ are computed as in $\Mod{R}_{\Z^m}$, namely, degreewise in $\kk$-modules for each multidegree $\vec{d} \in \Z^m$.
\end{proof}

\begin{defn}
A \Def{simplicial complex} on a vertex set $V$ is a collection $K \subseteq \PP{V} $ of subsets of $V$ that is downward closed, i.e.:
\[
\si \in K \quad \text{and} \quad \tau \subseteq \si \implies \tau \in K.
\]
The subsets $\si \in K$ are called \Def{faces} or \Def{simplices} of $K$. Subsets $\si \notin K$ are called \Def{missing faces} of~$K$. The \Def{dimension} of a face $\si \in K$ is its cardinality minus one: $\dim(\si) = \# \si - 1$. The \Def{$n$-skeleton} of $K$ is the subcomplex containing the faces of dimension at most $n$:
\[
\sk_n(K) = \{ \si \in K \mid \dim(\si) \leq n \}.
\]
\end{defn}

\begin{defn}\label{def:LocalModule}
Let $K \subseteq \PP{[m]}$ be a simplicial complex on the vertex set $[m]$. (We do not assume that all the elements $i \in [m]$ are $0$-simplices of $K$. In other words, $K$ may have ghost vertices.) A \Def{$K$-localized persistence module} $M$ consists of the following data:
\begin{enumerate}
\item For each missing face $\si \notin K$, a finitely generated $R_{\si}$-module $M_{\si}$.
\item For each $\si \subseteq \tau$ with $\si \notin K$ (and hence $\tau \notin K$), a map of $R_{\si}$-modules $\phy_{\si,\tau} \colon M_{\si} \to M_{\tau}$ such that the induced map of $R_{\tau}$-modules
\[
R_{\tau} \ot_{R_{\si}} M_{\si} \ral{\cong} M_{\tau}
\]
is an isomorphism. 
Here we view the $R_{\si}$-module $M_{\si}$ as $\tau^{-1} \N^m$-graded, to match the $\tau^{-1} \N^m$-grading of $M_{\tau}$.
\end{enumerate}
The transition maps are required to satisfy $\phy_{\si,\si} = \id$ and $\phy_{\tau,\upsilon} \circ \phy_{\si,\tau} = \phy_{\si,\upsilon}$ for all missing faces $\si \subseteq \tau \subseteq \upsilon$.

A morphism of $K$-localized persistence modules $f \colon M \to N$ consists of a map of $R_{\si}$-modules $f_{\si} \colon M_{\si} \to N_{\si}$ for each $\si \notin K$, compatible with the transition maps:
\[
\xymatrix{
M_{\si} \ar[d] \ar[r]^-{f_{\si}} & N_{\si} \ar[d] \\
M_{\tau} \ar[r]^-{f_{\tau}} & N_{\tau}. \\
}
\]
Let $\DD{K}$ denote the category of $K$-localized persistence modules. 
Let $\DD{K}_{\Z^m}$ denote the variant of $\DD{K}$ where all the $R_{\si}$-modules are $\Z^m$-graded rather than $\si^{-1}\N^m$-graded.
\end{defn}

\begin{ex}\label{ex:Extreme}
In the case $K = \emptyset = \sk_{-2} \De^{m-1}$, we have an equivalence $\DD{K} \cong \fgMod{R}$ given by $M \mapsto M_{\emptyset}$. The other extreme case is $K = \del \De^{m-1} = \sk_{m-2} \De^{m-1}$, which yields $\DD{K} = \fgMod{R_{[m]}} \cong \vect{\kk}$, cf.\ Remark~\ref{rem:GradedVect}.

For $m=2$, let us illustrate the case $K = \emptyset = \sk_{-2} \De^{1}$ with the free $R$-module on one generator in bidegree $\vec{a} = (a_1,a_2)$, namely:
\[
t^{\vec{a}} R = t_1^{a_1} t_2^{a_2} \kk[t_1,t_2] = (t_1^{a_1} t_2^{a_2}) \subset \kk[t_1,t_2]
\]
viewed as an %
ideal in $\kk[t_1,t_2]$. The corresponding $K$-localized persistence module $M$ consists of the following data:
\[
\xymatrix @C+3.8pc {
M_{\{2\}} = t_1^{a_1} \kk[t_1,t_2^{\pm}] \ar[r]^-{\phy_{\{2\}, \{1,2\}} = \text{ invert } t_1} & \kk[t_1^{\pm},t_2^{\pm}] = M_{\{1,2\}} \\
M_{\emptyset} = t_1^{a_1} t_2^{a_2} \kk[t_1,t_2] \ar[u]^{\phy_{\emptyset, \{2\}} = \text{ invert } t_2} \ar[r]^-{\phy_{\emptyset, \{1\}} = \text{ invert } t_1} & t_2^{a_2} \kk[t_1^{\pm},t_2] = M_{\{1\}} \ar[u]_{\phy_{\{1\}, \{1,2\}} = \text{ invert } t_2} \\
}
\]
where the transition maps $\phy_{\si,\tau}$ are the canonical inclusions, as illustrated in Figure~\ref{fig:AllLocaliz}. 
\end{ex}

\begin{figure}[H]
\begin{adjustbox}{width=0.55\textwidth}
\begin{tikzpicture}[scale=0.8]
 \begin{axis}[
axis lines=middle, 
ticks=none, 
axis equal, 
xmin=-0.5, xmax=5, ymin=-1, ymax=5]
\draw [blue, thick] (2,6) -- (2,-1);
\fill [blue, opacity=0.1] (2,-1) rectangle (6,6);
\node [below left] at (2,0) {$a_1$};
\draw [blue, dotted] (2,3) -- (0,3);
\node [left] at (0,3) {$a_2$};
 \end{axis}
\draw [->] (7.0,3) -- (8.6,3);
\node [above] at (7.8,3) {$\text{invert } t_1$};
\end{tikzpicture}
\begin{tikzpicture}[scale=0.8]
 \begin{axis}[
axis lines=middle, 
ticks=none, 
axis equal, 
xmin=-0.5, xmax=5, ymin=-1, ymax=5]
\fill [blue, opacity=0.1] (-1,-1) rectangle (6,6);
\draw [blue, dotted] (2,3) -- (2,0);
\node [below] at (2,0) {$a_1$};
\draw [blue, dotted] (2,3) -- (0,3);
\node [left] at (0,3) {$a_2$};
 \end{axis}
\end{tikzpicture}
\end{adjustbox}
\begin{adjustbox}{width=0.55\textwidth}
\begin{tikzpicture}[scale=0.8]
 \begin{axis}[
axis lines=middle, 
ticks=none, 
axis equal, 
xmin=-0.5, xmax=5, ymin=-1, ymax=5]
\draw [blue, fill=blue] (2,3) circle (\VertexSize);
\draw [blue, thick] (2,6) -- (2,3) -- (6,3);
\fill [blue, opacity=0.1] (2,3) rectangle (6,6);
\draw [blue, dotted] (2,3) -- (2,0);
\node [below] at (2,0) {$a_1$};
\draw [blue, dotted] (2,3) -- (0,3);
\node [left] at (0,3) {$a_2$};
 \end{axis}
\draw [->] (3,5.8) -- (3,7.2);
\node [right] at (3,6.5) {$\text{invert } t_2$};
\draw [->] (7.0,3) -- (8.6,3);
\node [above] at (7.8,3) {$\text{invert } t_1$};
\end{tikzpicture}
\begin{tikzpicture}[scale=0.8]
 \begin{axis}[
axis lines=middle, 
ticks=none, 
axis equal, 
xmin=-0.5, xmax=5, ymin=-1, ymax=5]
\draw [blue, thick] (-1,3) -- (6,3);
\fill [blue, opacity=0.1] (-1,3) rectangle (6,6);
\draw [blue, dotted] (2,3) -- (2,0);
\node [below] at (2,0) {$a_1$};
\node [below left] at (0,3) {$a_2$};
 \end{axis}
\draw [->] (3,5.8) -- (3,7.2);
\node [right] at (3,6.5) {$\text{invert } t_2$};
\end{tikzpicture}
\end{adjustbox}
\caption{The $R$-module $t^{\vec{a}} R$ together with its localizations.}
\label{fig:AllLocaliz}
\end{figure}

\begin{ex}\label{ex:2dim}
Consider the case $m=2$ and $K = \{ \emptyset \} = \sk_{-1} \De^1$. A $K$-localized persistence module $M$ consists of modules
\[
\xymatrix{
M_{\{2\}} \ar[r]^-{\phy_1} & M_{\{1,2\}} \\
\ar@{--}@[green][u] \ar@{--}@[green][r] & M_{\{1\}} \ar[u]_{\phy_2} \\
}
\]
where for each $i \in \{ 1,2 \}$, $M_{\{i\}}$ is a $R[t_i^{-1}]$-module, $\phy_{\{3-i\}} \colon M_{\{i\}} \to M_{\{1,2\}}$ is a map of \mbox{$R[t_i^{-1}]$-modules}, and the induced map of $R[t_1^{-1}, t_2^{-1}]$-modules
\[
R[t_1^{-1}, t_2^{-1}] \ot_{R[t_i^{-1}]} M_{\{i\}} \ral{\cong} M_{\{1,2\}}
\]
is an isomorphism. We used the notation
\[
\phy_i \dfn \phy_{[m] \setminus \{i\}, [m]} \colon M_{[m] \setminus \{i\} } \to M_{[m]},
\]
so that $\phy_i$ denotes the localization map inverting $t_i$. We will write $\si_i \dfn [m] \setminus \{i\}$.
\end{ex}

\begin{ex}
Consider the case $m=3$ and $K = \sk_0 \De^2 = \{ \emptyset, \{1\}, \{2\}, \{3\} \}$. A $K$-localized persistence module $M$ can be displayed as a diagram
\[
\xymatrix @C-0.7pc @R-0.7pc {
& M_{\si_1} \ar[rr]^-{\phy_{1}} & & M_{[3]} \\
\ar@{--}@[green][ur] \ar@{--}@[green][rr] & & M_{\si_2} \ar[ur]^-{\phy_{2}} & \\
& \ar@{--}@[green][uu] \ar@{--}@[green][rr] & & M_{\si_3}. \ar[uu]_{\phy_{3}} \\
\ar@{--}@[green][uu] \ar@{--}@[green][ur] \ar@{--}@[green][rr] & & \ar@{--}@[green][uu] \ar@{--}@[green][ur] & \\
}
\]
\end{ex}

\begin{rem}\label{rem:GradedVect}
A graded $\kk[t^{\pm}]$-module $M$ corresponds to a $\kk$-vector space
\[
M(0) \cong \colim \left( \cdots \to M(d-1) \ral{t} M(d) \ral{t} M(d+1) \to \cdots \right) 
\]
In fact, the functor taking the degree $0$ part
\[
\Mod{\kk[t^{\pm}]} \to \Vect{\kk}
\]
is an equivalence of categories. This is in stark contrast with an \emph{ungraded} $\kk[t^{\pm}]$-module, which consists of a $\kk$-vector space $V$ equipped with an automorphism $\mu_t \colon V \ral{\cong} V$. 
Likewise, for any $m \geq 1$, a $\Z^m$-graded $R_{[m]}$-module $M$ corresponds to the $\kk$-vector space $M(\vec{0})$ in multidegree $\vec{0} \in \Z^m$. Because of this, we will write %
\[
\dim_{\kk} M \dfn \dim_{\kk} M(\vec{0}).
\]
\end{rem}

\begin{defn}
For $\si \subseteq [m]$ and an $R_{\si}$-module $M$, the \Def{rank} of $M$ is
\[
\rank M \dfn \dim_{\kk} R_{[m]} \ot_{R_{\si}} M.
\]
For a simplicial complex $K$ on the vertex set $[m]$ and $M$ a $K$-localized persistence module, the \Def{rank} of $M$ is
\[
\rank M \dfn \dim_{\kk} M_{[m]},
\]
which equals $\rank(M_{\si})$ for any missing face $\si \notin K$.
\end{defn}
See also the discussion of rank in \cite{HarringtonOST19}*{\S 3}.

\begin{lem}\label{lem:PointwiseExact}
The category $\DD{K}$ is an abelian category. Kernels and cokernels in $\DD{K}$ are computed pointwise, i.e., for each $\si \notin K$, we have
\begin{align*}
(\ker f)_{\si} &= \ker(f_{\si}) \\
(\coker f)_{\si} &= \coker(f_{\si}).
\end{align*}
In particular, a sequence
\[
0 \to M \to N \to P \to 0
\]
in $\DD{K}$ is exact if and only if for each $\si \notin K$, the sequence of $R_{\si}$-modules %
\[
0 \to M_{\si} \to N_{\si} \to P_{\si} \to 0
\]
is exact.
\end{lem}

\begin{proof}
The assignment $\si \mapsto R_{\si}$ defines a diagram of commutative (graded) rings, indexed by the missing faces $\si \in \PP{[m]} \setminus K$. Modules over a diagram of rings form a complete and cocomplete abelian category, with limits and colimits computed pointwise \cite{KashiwaraS06}*{Theorem~18.1.6}. Since each ring $R_{\si}$ is Noetherian, pointwise finitely generated modules form an abelian subcategory, which we denote $\mathcal{L}^{\text{no loc}}(K)$. %
Then $\DD{K} \subset \mathcal{L}^{\text{no loc}}(K)$ is the full subcategory of objects satisfying the localization condition $R_{\tau} \ot_{R_{\si}} M_{\si} \cong M_{\tau}$, which is an abelian subcategory.
\end{proof}

\section{An equivalence of categories}

\begin{defn}
Given a simplicial complex $K$ on the set of vertices $[m]$, consider the functor
\[
L_K \colon \fgMod{R} \to \DD{K}
\]
that sends an $R$-module $M$ to the $K$-localized persistence module $L_K(M)$ given by the localizations
\[
L_K(M)_{\si} = R_{\si} \ot_R M  
\]
and the canonical localization maps $R_{\si} \ot_{R} M \to R_{\tau} \ot_{R} M$ for missing faces $\si \subseteq \tau$. We call $L_K$ the \Def{$K$-localization functor}.
\end{defn}

\begin{lem}
The functor $L_K$ is exact.
\end{lem}

\begin{proof}
This follows from Lemma~\ref{lem:PointwiseExact} and the fact that each localization functor $R_{\si} \otimes_R -$ is exact, in other words, $R_{\si}$ is flat as an $R$-module, by (the graded analogue of) \cite{AtiyahM69}*{Corollary~3.6}.
\end{proof}

\begin{lem}\label{lem:Truncation}
For $\si \subseteq [m]$, let %
\[
\io \colon \Mod{R} \inj \Mod{R}_{\si^{-1} \N^m}
\]
denote the inclusion of the full subcategory of $\N^m$-graded $R$-modules.
\begin{enumerate}
\item\label{item:TruncationAdjoint} The truncation functor
\[
\tau_{\geq 0} \colon \Mod{R}_{\si^{-1} \N^m} \to \Mod{R}
\]
given by
\[
(\tau_{\geq 0} M)(\vec{d}) = M(\vec{d}) \quad \text{for } \vec{d} \in \N^m
\]
is right adjoint to $\io$.
\item\label{item:TruncationFinite} $\tau_{\geq 0}$ sends free $R$-modules to free $R$-modules and finitely generated $R$-modules to finitely generated $R$-modules. Hence the resulting functor
\[
\tau_{\geq 0} \colon \fgMod{R}_{\si^{-1} \N^m} \to \fgMod{R}
\]
is right adjoint to the inclusion functor.
\end{enumerate}
\end{lem}

\begin{proof}
\Eqref{item:TruncationAdjoint} Let $M$ be an $\N^m$-graded $R$-module and $N$ a $\si^{-1} \N^m$-graded $R$-module. A map of $\si^{-1} \N^m$-graded \mbox{$R$-modules} $f \colon \io M \to N$ 
is determined by maps of $\kk$-modules $M(\vec{d}) \to N(\vec{d})$ for each multidegree $\vec{d} \in \N^m$, since $M$ is zero outside of that range. 
Since $R$ is $\N^m$-graded, 
$f$ corresponds to the data of a map of $R$-modules $M \to \tau_{\geq 0} N$, and this bijective correspondence is natural in $M$ and $N$.

\Eqref{item:TruncationFinite} For $x \in M$ in multidegree $\abs{x} = \vec{d}$, let $\ceil{x}$ denote the shift of $x$ obtained by replacing the negative degrees $d_i$ with zeroes, that is:
\[
\ceil{x} \dfn t^{\de} x \quad \text{where } \de = \sup \{ \vec{0}, -\abs{x} \} = -\inf \{ \vec{0}, \abs{x} \}
\]
using the componentwise partial order on $\si^{-1} \N^m$. 
For example, with $\abs{x} = (-3,5)$, we obtain $\ceil{x} = t_1^3 x$ %
in bidegree $(0,5)$. Now 
for a free $\si^{-1} \N^m$-graded $R$-module $R \lan x \ran$ %
on one generator $x$ in multidegree $\abs{x} \in \si^{-1} \N^m$, the truncation is
\[
\tau_{\geq 0}(R\lan x \ran) \cong R \lan \ceil{x} \ran \quad \text{in } \fgMod{R}.
\]
By Lemma~\ref{lem:ZmGraded}, the functor $\tau_{\geq 0}$ preserves direct sums and cokernels, which yields the result.
\end{proof}

\begin{lem}\label{lem:Delocalization}
\begin{enumerate}
\item\label{item:RightAdj} The functor $L_K \colon \fgMod{R} \to \DD{K}$ admits a right adjoint $\rho_K \colon \DD{K} \to \fgMod{R}$.
\item\label{item:Counit} The counit
\[
\ep \colon L_K \rho_K (M) \ral{\cong} M
\]
is an isomorphism for all $M$ in $\DD{K}$.
\end{enumerate}
\end{lem}

\begin{proof}
\Eqref{item:RightAdj} We claim that the right adjoint $\rho_K$ is given by
\[
\rho_K (M) = \tau_{\geq 0} \left( \lim_{\si \in \PP{[m]} \setminus K} M_{\si} \right),
\]
where the limit over the missing faces $\si$ is computed in $\Mod{R}_{\Z^m}$. Indeed, %
for any object $P$ in $\fgMod{R}$, there is a natural bijection of hom-sets 
\begin{flalign*}
&& \fgMod{R} \left( P, \rho_K(M) \right) &= \fgMod{R} \left( P, \tau_{\geq 0} \left( \lim_{\si \in \PP{[m]} \setminus K} M_{\si} \right) \right) && \\
&& &= \Mod{R}_{\Z^m} \left( P, \lim_{\si \in \PP{[m]} \setminus K} M_{\si} \right) && \text{by Lemma~\ref{lem:Truncation}} \\
&& &\cong \Mod{R}_{\Z^m}^{\PP{[m]} \setminus K} \left( \cst(P), M \right) && \text{hom in $(\PP{[m]} \setminus K)$-shaped diagrams} \\
&& & && \text{from the constant diagram on $P$} \\
&& &\cong \mathcal{L}^{\text{no loc}}(K) \left( L_K (P), M \right) && \text{see Lemma~\ref{lem:PointwiseExact}} \\ %
&& &= \DD{K} \left( L_K(P), M \right). &&
\end{flalign*}
The next-to-last isomorphism uses the fact that for each missing face $\si \notin K$, the $R_{\si}$-module $M_{\si}$ is $\si$-local (i.e., $t_i$ acts invertibly on $M_{\si}$ for all $i \in \si$), which yields
\[
\Mod{R}_{\Z^m}(P, M_{\si}) \cong \Mod{R_{\si}}_{\Z^m}(R_{\si} \ot_R P, M_{\si}).
\]
\Eqref{item:Counit} The counit $\ep \colon L_K \rho_K (M) \to M$ is given in position $\si \in \PP{[m]} \setminus K$ by the map of $R_{\si}$-modules
\[
\xymatrix @R-1pc {
\left( L_K \rho_K (M) \right)_{\si} \ar@{=}[d] \ar[r]^-{\ep_{\si}} & M_{\si}, \\
R_{\si} \ot_{R} \rho_K (M)
}
\]
which we will show is an isomorphism. First, the $R$-module $\lim_{\si \in \PP{[m]} \setminus K} M_{\si}$ can have elements of multidegree $\vec{d}$ with some $d_i < 0$ only if $i$ belongs to all the missing faces $\si \in \PP{[m]} \setminus K$, so that $t_i$ already acts invertibly on all the $M_{\si}$. Thus we may assume without loss of generality that the missing faces have empty intersection $\bigcap_{\PP{[m]} \setminus K} \si = \emptyset$, in which case the truncation does nothing:
\[
\tau_{\geq 0} \left( \lim_{\si \in \PP{[m]} \setminus K} M_{\si} \right) = \lim_{\si \in \PP{[m]} \setminus K} M_{\si}.
\]
Now fix a missing face $\si$. Since localization $R_{\si} \ot_R -$ preserves finite limits, 
we have: 
\begin{equation}\label{eq:LocalizationLim}
R_{\si} \ot_R  \left( \lim_{\tau \in \PP{[m]} \setminus K} M_{\tau} \right) \cong \lim_{\tau \in \PP{[m]} \setminus K} (R_{\si} \ot_R M_{\tau}).
\end{equation}
For any other missing face $\si'$, the transition map
\[
\phy_{\si', \si \cup \si'} \colon M_{\si'} \to M_{\si \cup \si'}
\]
is inverting the $t_i$ for $i \in \si \setminus \si'$, in particular becomes an isomorphism after $\si$-localizing. Hence the limit \eqref{eq:LocalizationLim} is computed on the subdiagram past the position $\si$:
\begin{align*}
\lim_{\tau \in \PP{[m]} \setminus K} (R_{\si} \ot_R M_{\tau}) &\cong \lim_{\substack{\tau \in \PP{[m]} \setminus K\\ \si \subseteq \tau}} (R_{\si} \ot_R M_{\tau}) & \\
&\cong \lim_{\substack{\tau \in \PP{[m]} \setminus K\\ \si \subseteq \tau}} M_{\tau} & \text{since $M_{\tau}$ is already $\si$-local} \\
&\cong M_{\si}
\end{align*}
since $M_{\si}$ is initial among said subdiagram. 
\end{proof}

\begin{rem}
Being a right adjoint, the functor $\rho_K \colon \DD{K} \to \fgMod{R}$ preserves limits and thus is left exact. However, $\rho_K$ need \emph{not} be right exact. For example, consider the case $m=2$ with $K = \{ \emptyset \}$, as in Example~\ref{ex:2dim}. The quotient map $q \colon R \surj R/(t_1 t_2)$ in $\fgMod{R}$ yields the epimorphism in $\DD{K}$
\[
L_K(q) \colon L_K (R) \surj L_K(R/ (t_1 t_2)).
\]
However, the map in $\fgMod{R}$ 
\[
\rho_K L_K(q) \colon \rho_K L_K (R) \surj \rho_K L_K(R/ (t_1 t_2))
\]
is \emph{not} an epimorphism. Its source is $\rho_K L_K(R) \cong R$ and its target is
\begin{align*}
\rho_K L_K(R/ (t_1 t_2)) &= \tau_{\geq 0} \lim \left( \kk[t_1, t_2^{\pm}] / (t_1) \to 0 \leftarrow \kk[t_1^{\pm}, t_2] / (t_2) \right) \\
&= \tau_{\geq 0} \left( \kk[t_1, t_2^{\pm}] / (t_1) \op \kk[t_1^{\pm}, t_2] / (t_2) \right) \\
&= \kk[t_1, t_2] / (t_1) \op \kk[t_1, t_2] / (t_2) \\
&= R / (t_1) \op R / (t_2).
\end{align*}
The map $\rho_K L_K(q)$ is the map
\[
\xymatrix{
R \ar[r]^-{\smat{q_1 \\ q_2}} & R/(t_1) \op R/(t_2) \\
}
\]
with coordinates the two quotient maps, which is not surjective in bidegree $(0,0)$. 
\end{rem}

\begin{prop}\label{pr:SerreQuotient}
For any simplicial complex $K$ on the vertex set $[m]$, the induced functor $\ol{L_K}$ out of the Serre quotient 
is an equivalence of categories:
\[
\xymatrix{
\fgMod{R} \ar[d]_{q} \ar[r]^-{L_K} & \DD{K}. \\
\fgMod{R} / \ker(L_K) \ar@{-->}[ur]_{\ol{L_K}}^{\simeq} & \\
}
\]
\end{prop}

\begin{proof}
This follows from Lemmas~\ref{lem:Delocalization} and \ref{lem:SerreEquiv}.
\end{proof}

We will argue the $\Z^m$-graded analogue of Proposition~\ref{pr:SerreQuotient}. The localization functor
\[
L_K \colon \fgMod{R}_{\Z^m} \to \DD{K}_{\Z^m}
\]
is defined the same way, though the construction of a right adjoint as in Lemma~\ref{lem:Delocalization} does \emph{not} work in the $\Z^m$-graded case. Nevertheless, we will reduce the problem to the $\N^m$-graded case by shifting the degrees.

\begin{prop}\label{pr:SerreQuotientZm}
For any simplicial complex $K$ on the vertex set $[m]$, the induced functor $\ol{L_K}$ out of the Serre quotient 
is an equivalence of categories:
\[
\xymatrix{
\fgMod{R}_{\Z^m} \ar[d]_{q} \ar[r]^-{L_K} & \DD{K}_{\Z^m}. \\
\fgMod{R}_{\Z^m} / \ker(L_K) \ar@{-->}[ur]_{\ol{L_K}}^{\simeq} & \\
}
\]
\end{prop}

\begin{proof}
By Lemma~\ref{lem:SerreFaithful}, $\ol{L_K}$ is faithful. We now argue that $\ol{L_K}$ is essentially surjective. 

Given a multidegree $\vec{e} \in \Z^m$, let $T_{\vec{e}} \colon \fgMod{R}_{\Z^m} \ral{\cong} \fgMod{R}_{\Z^m}$ denote the isomorphism of categories that shifts the modules by $\vec{e}$, that is:
\[
(T_{\vec{e}} M)(\vec{d}) = M(\vec{d} - \vec{e}).
\]
Also denote by $T_{\vec{e}} \colon \DD{K}_{\Z^m} \ral{\cong} \DD{K}_{\Z^m}$ the induced isomorphism of categories. 
Now 
let $M$ be an object of $\DD{K}_{\Z^m}$. There is a large enough multidegree $\vec{e} \in \Z^m$ such that the shift $T_{\vec{e}} M$ lies in the full 
subcategory $\DD{K}$. By Proposition~\ref{pr:SerreQuotient}, there is an $R$-module $A$ satisfying $\ol{L_K}(A) \cong T_{\vec{e}} M$. Undoing the shift yields
\[
\ol{L_K}(T_{-\vec{e}} A) \cong T_{-\vec{e}} \ol{L_K}(A) \cong T_{-\vec{e}} T_{\vec{e}} M = M.
\]
The same argument works for maps $f \colon M \to N$ in $\DD{K}_{\Z^m}$, showing that $\ol{L_K}$ is full. 
\end{proof}

\begin{ex}
In the case $m=2$ still with $K = \{ \emptyset \} = \sk_{-1} \De^1$, %
consider the %
ideal $(t_1, t_2^2) \subset R$ and the quotient $R$-module
\[
M = (t_1, t_2^2) / (t_1^3, t_1^2 t_2^3).
\] 
The delocalization of the localization of $M$ is
\[
\rho_{K} L_{K}(M) = R/(t_1^2) 
\]
as illustrated in Figure~\ref{fig:Delocaliz}. The unit map $\eta \colon M \to \rho_{K} L_{K}(M)$ fills in the area shaded in green and quotients out the area shaded in red.

\begin{figure}[H]
\begin{adjustbox}{width=0.6\textwidth}
\begin{tikzpicture}[scale=1]
 \begin{axis}[
axis lines=middle, 
ticks=none, 
axis equal, 
xmin=-0.2, xmax=5, ymin=-1, ymax=5]
\draw [blue, fill=blue] (0,2) circle (\VertexSize);
\draw [blue, fill=blue] (1,0) circle (\VertexSize);
\draw [red, fill=red] (3,0) circle (\VertexSize);
\draw [red, fill=red] (2,3) circle (\VertexSize);
\draw [blue, thick] (0,6) -- (0,2) -- (1,2) -- (1,0) -- (3,0);
\draw [blue, dotted] (3,0) -- (3,3) -- (2,3) -- (2,6);
\fill [blue, opacity=0.1] (0,6) -- (0,2) -- (1,2) -- (1,0) -- (3,0) -- (3,3) -- (2,3) -- (2,6) -- (0,6);
\node at (3,-0.8) {$M$};
 \end{axis}
\draw [->] (6.8,3) -- (7.5,3);
\node [above] at (7.1,3) {$\eta$};
\end{tikzpicture}
\hspace{3ex}
\begin{tikzpicture}[scale=1]
 \begin{axis}[
axis lines=middle, 
ticks=none, 
axis equal, 
xmin=-0.2, xmax=5, ymin=-1, ymax=5]
\draw [blue, fill=blue] (0,0) circle (\VertexSize);
\draw [red, fill=red] (2,0) circle (\VertexSize);
\draw [blue, thick] (0,6) -- (0,0) -- (2,0);
\fill [blue, opacity=0.1] (0,0) rectangle (2,6);
\draw [blue, dotted] (2,0) -- (2,6);
\fill [green, opacity=0.1] (0,0) rectangle (1,2);
\fill [red, opacity=0.1] (2,0) rectangle (3,3);
\node at (2.7,-0.8) {$\rho_{K} L_{K}(M)$};
 \end{axis}
\end{tikzpicture}
\end{adjustbox}
\caption{Delocalization of the localization of the $\kk[t_1,t_2]$-module $M$.}
\label{fig:Delocaliz}
\end{figure}
\end{ex}

\section{Torsion pair for localized persistence modules}\label{sec:TorsionTh}

In this section and the next two we restrict to the simplicial complex 
\[
K_m \coloneqq \sk_{m-3} \De^{m-1} = \{ \sigma \subseteq [m] \mid \#\sigma \leq m-2\}, 
\]
the codimension $2$ skeleton of $\Delta^{m-1}$. We construct two subcategories of 
$\DD{K_m}$ , a torsion subcategory $\cat{T}$ and a torsion-free subcategory $\cat{F}$ which together give a torsion pair. We show that objects in $\DD{K_m}$ break into a direct sum of their torsion and torsion-free parts. In the following two sections, we break the torsion objects into direct sums of indecomposables, and also the torsion-free objects in the case $m=2$. 

Background on torsion pairs (also called torsion theories) can be found in \cite{Dickson66}, \cite{Popescu73}*{\S 4.8}, or \cite{Borceux94v2}*{\S 1.12}.

\begin{defn}
Let $\cat{A}$ be an abelian category. A \Def{torsion pair}, $(\cat{T}, \cat{F})$ 
of $\cat{A}$ is a pair of 
full %
subcategories $\cat{T}, \cat{F} \subseteq \cat{A}$ both closed under isomorphisms, such that:
\begin{enumerate}
\item for every object $M\in \cat{A}$ there is a short exact sequence 
\begin{equation}\label{eq:SEStorsion}
\xymatrix{
0 \ar[r] & T(M) \ar[r] & M \ar[r] & F(M) \ar[r] & 0 \\
}
\end{equation}
where $T(M)\in \cat{T}$ and $F(M)\in \cat{F}$, and 
\item $\Hom(T,F)=0$ for any $T\in \cat{T}$ and $F\in \cat{F}$.
\end{enumerate}
\end{defn}

The subcategory $\cat{T}$ is called the \emph{torsion} class and its objects torsion objects and $\cat{F}$ is the \emph{torsion-free} class whose objects are also called torsion-free.

The short exact sequence~\eqref{eq:SEStorsion} is unique up to isomorphism \cite{Borceux94v2}*{Proposition~1.12.4} and may be chosen to be functorial in $M$ \cite{BauerBOS20}*{Remark~2.4}. 

For $i\in [m]$ let $\sigma_i=[m]\setminus \{ i\}$. 
For $M\in  \fgMod{R_{\sigma_i}}$, $T_i(M)$ is the submodule of torsion elements in $M$, so 
$x\in T_i(M)$ if for some $k$, $t_i^k x = 0$. 
We define two full subcategories $\cat{T}_i, \cat{F}_i\subset \fgMod{R_{\sigma_i}}$ by 
\begin{align*} 
\cat{T}_i &= \{ M \mid T_i(M)=M \} = \{ M \mid M\otimes R_{[m]}=0 \} \\
\cat{F}_i &= \{ M \mid T_i(M)=0 \}.
\end{align*}
Note that $T_i$ %
defined an endofunctor of $\fgMod{R_{\sigma_i}}$ 
and we define the endofunctor $F_i$ by $F_i(M)=M/T_i(M)$. 
We record some of this in a proposition. 

\begin{prop}\label{pr:weareTT}
For every $M$ in $\fgMod{R_{\sigma_i}}$ we have a short exact sequence 
\[
0\to T_i(M)\to M \to F_i(M) \to 0
\]
where $T_i(M)\in \cat{T}_i$ and $F_i(M)\in \cat{F}_i$.

Also for any $A \in \cat{T}_i$ and $B\in \cat{F}_i$, $\Hom(A,B)=0$, hence 
$(\cat{T}_i, \cat{F}_i)$ is a torsion pair. 
\end{prop}

We now recall 
the decomposition of objects 
in $\fgMod{R_{\sigma_i}}$ into indecomposables, namely interval modules. We use %
interval notation 
\begin{align*}
&[a,b)_{i} = t_i^{a} R_{\sigma_i} / t_i^{b} = \coker (t_i^{a} R_{\sigma_i} \ral{t_i^{b-a}} t_i^{a} R_{\sigma_i}) \\
&[a,\infty)_{i} = t_i^{a} R_{\sigma_i}.
\end{align*}

\begin{prop}\label{pr:wesplit}
\begin{enumerate}
\item The $R_{\sigma_i}$-modules %
$[a,b)_{i}$ and $[a,\infty)_{i}$ 
are indecomposable. 
\item For any $M$ the short exact sequence in Proposition~\ref{pr:weareTT} splits, and thus 
$M\cong T_i(M)\oplus F_i(M)$. 
\item For any torsion module $M\in \cat{T}_i$, there is a decomposition 
\[
M\cong \bigoplus_j [a_j,b_j)_{i}.
\]
For any torsion-free module $M\in \cat{F}_i$, there is a decomposition 
\[
M\cong \bigoplus_j [a_j,\infty)_{i}.
\]
\end{enumerate}
\end{prop}

\begin{proof}
The inclusion $\kk[t_i] \to R_{\sigma_i}$ induces the restriction functor
$\fgMod{R_{\sigma_i}} \to \fgMod{\kk[t_i]}$, which is an equivalence of categories, cf.\ Remark~\ref{rem:GradedVect}. The claims then follow from the same %
results in $\fgMod{\kk[t_i]}$, by the structure theorem for finitely generated graded modules over a graded principal ideal domain, as observed in  \cite{CarlssonZ05}*{Theorem~2.1}. %
\end{proof}

We define these two subcategories of $\DD{K_m}$: 
\begin{align*}
&\cat{T}=\{ M \mid M_{\sigma_i} \in \cat{T}_i {\rm \ for \ all \ } i\in [m] \} = \{ M\mid M_{[m]}=0\} \\
&\cat{F}=\{ M \mid M_{\sigma_i} \in \cat{F}_i {\rm \ for \ all \ } i\in [m] \}.
\end{align*}
Corresponding to $\cat{T}$ and $\cat{F}$ we have endofunctors 
$T,F \colon \DD{K_m} \to \DD{K_m}$ defined by
\[ 
T(M)_{\sigma}=
\begin{cases}
T_i(M_{\sigma_i}) & {\rm if \ } \sigma=\sigma_i \\
0 & {\rm if \ } \sigma=[m]
\end{cases}
\]
and 
\[
F(M)_{\sigma}=M_{\sigma}/T(M)_{\sigma}.
\]

\begin{prop}\label{pr:allareTT}
\begin{enumerate}
\item \label{item:allareTT} For any $M$ in $\DD{K_m}$, $T(M)\in \cat{T}$, 
$F(M)\in \cat{F}$ and 
we have a short exact sequence 
\[
0\to T(M)\to M \to F(M) \to 0.
\]
Also for any $A \in \cat{T}$ and $B\in \cat{F}$, $\Hom(A,B)=0$, hence 
$(\cat{T}, \cat{F})$ is a torsion pair in $\DD{K_m}$. 
\item \label{item:allsplit} For any $M$ in $\DD{K_m}$, the short exact sequence in part~\Eqref{item:allareTT} splits, and thus 
$M\cong T(M)\oplus F(M)$. 
\end{enumerate}
\end{prop}

\begin{proof}
\Eqref{item:allareTT} It is clear that %
$T$ lands in $\cat{T}$. That 
$F$ lands in $\cat{F}$ follows from the fact that $F_i$ lands in $\cat{F}_i$ 
by Proposition~\ref{pr:weareTT} and since cokernels are computed pointwise (Lemma~\ref{lem:PointwiseExact}). That the sequence is exact and that $M/T(M)$ is torsion-free both follow since cokernels are computed pointwise.  
Finally $\Hom(A,B)=0$ since for every $\sigma \not\in K_m$, 
$\Hom(A_{\sigma},B_{\sigma})=0$ by Proposition~\ref{pr:weareTT}.

\Eqref{item:allsplit} We construct a retraction of the inclusion
\[
I\colon T(M) \to M.
\]
For $i\in [m]$, by Proposition~\ref{pr:wesplit}, 
$I_{\sigma_i}\colon T(M)_{\sigma_i} \to M_{\sigma_i}$ has a retraction
$r_{\sigma_i}\colon M_{\sigma_i}\to T(M)_{\sigma_i}$. 
Since $T(M)_{[m]}=0$, letting $r_{[m]}=0$ defines a retraction to $I$. 
\end{proof}

\section{Decomposition of torsion objects}\label{sec:Torsion}

We want to describe some objects in $\DD{K_m}$ that will be our indecomposables. 

\begin{nota}\label{nota:Interval}
For $a<b\in \N$ and $i\in [m]$, we view the $R_{\si_i}$-module $[a,b)_i$ as an object of $\DD{K_m}$ also denoted (by abuse of notation)
$[a,b)_i\in \DD{K_m}$, defined by 
\[
[a,b)_i(\sigma)=
\begin{cases}
[a,b)_{i} &\text{if } \sigma=\sigma_i \\
0 &\text{if } \sigma=\sigma_j, \ j \neq i \\
0 &\text{if } \sigma=[m]. \\
\end{cases}
\]
\end{nota}
The objects $[a,b)_i$ are illustrated in Figure~\ref{fig:Strips} in the case $m=2$.

\begin{lem}\label{lem:LocalizTorsion}
There is an isomorphism in $\DD{K_m}$
\[
[a,b)_i \cong L_{K_m} \left( t_i^a R / t_i^b R \right).
\] 
\end{lem}

\begin{lem}\label{lem:IndecompoTorsion}
The objects $[a,b)_i$ in $\cat{T}$ are indecomposable.
\end{lem}

\begin{prop}\label{pr:DecompoTorsion}
Each torsion object $M\in \cat{T}$ decomposes as a direct sum of objects of the form $[a,b)_i$. 
\end{prop}

\begin{proof}
For each $i \in [m]$, the $R_{\si_i}$-module admits an interval decomposition
\[
M_{\sigma_i} \cong \bigoplus_j [a_j, b_j)_i.
\]
Since $M \in \cat{T}$ is torsion, it satisfies $M_{[m]}=0$. Thus the decompositions assemble into a decomposition 
$M \cong \bigoplus_i \bigoplus_j [a_j, b_j)_i$, 
where each interval $[a_j, b_j)_i$ is viewed as an object of $\DD{K_m}$ as in Notation~\ref{nota:Interval}.
\end{proof}

\section{Decomposition of torsion-free objects in dimension \texorpdfstring{$m=2$}{2}}\label{sec:Torsionfree}

In this section, we investigate the structure of the torsion-free objects $\cat{F} \subset \DD{K_m}$ defined after Proposition~\ref{pr:wesplit}.

\begin{nota}\label{nota:Quadrant}
For a multidegree $\vec{a} = (a_1, \cdots , a_m) \in \N^m$, define $[\vec{a},\infty) \in \DD{K_m}$ by 
\[
[\vec{a},\infty)_{\sigma} = 
\begin{cases}
[a_i,\infty)_{i} &\text{if } \sigma = \sigma_i \\
R_{[m]} &\text{if } \sigma = [m]. \\
\end{cases}
\]
The transition map $\phy_i \colon [a_i,\infty)_{i} \to R_{[m]}$ is the inclusion of $R_{\si_i}$-submodule.
\end{nota}

\begin{figure}[H]
\begin{adjustbox}{width=\textwidth}
\begin{tikzpicture}[scale=1]
 \begin{axis}[
axis lines=middle, 
ticks=none, 
axis equal, 
xmin=-0.5, xmax=5, ymin=-1, ymax=5]
\draw [blue, thick] (1,-1) -- (1,6);
\draw [blue, dashed] (3,-1) -- (3,6);
\fill [blue, opacity=0.1] (1,-1) rectangle (3,6);
\node [below left] at (1,0) {$a$};
\node [below right] at (3,0) {$b$};
 \end{axis}
\node at (3.5,-0.7) {$[a, b)_1$};
\end{tikzpicture}
\hspace{1ex}
\begin{tikzpicture}[scale=1]
 \begin{axis}[
axis lines=middle, 
ticks=none, 
axis equal, 
xmin=-0.5, xmax=5, ymin=-1, ymax=5]
\draw [blue, thick] (-1,1) -- (6,1);
\draw [blue, dashed] (-1,3) -- (6,3);
\fill [blue, opacity=0.1] (-1,1) rectangle (6,3);
\node [below left] at (0,1) {$a$};
\node [above left] at (0,3) {$b$};
 \end{axis}
\node at (3.5,-0.7) {$[a,b)_2$};
\end{tikzpicture}
\hspace{1ex}
\begin{tikzpicture}[scale=1]
 \begin{axis}[
axis lines=middle, 
ticks=none, 
axis equal, 
xmin=-0.5, xmax=5, ymin=-1, ymax=5]
\draw [blue, fill=blue] (2,2) circle (\VertexSize);
\draw [blue, thick] (2,6) -- (2,2) -- (6,2);
\fill [blue, opacity=0.1] (2,2) rectangle (6,6);
\draw [blue, dotted] (2,2) -- (2,0);
\node [below] at (2,0) {$a_1$};
\draw [blue, dotted] (2,2) -- (0,2);
\node [left] at (0,2) {$a_2$};
 \end{axis}
\node at (3.5,-0.7) {$[\vec{a}, \infty)$};
\end{tikzpicture}
\end{adjustbox}
\caption{Some indecomposable objects of $\DD{K_2}$.}
\label{fig:Strips}
\end{figure}

\begin{lem}\label{lem:LocalizFree}
For any %
$\vec{a} \in \N^m$, there is an isomorphism in $\DD{K_m}$
\[
[\vec{a},\infty ) \cong L_{K_m} (t^{\vec{a}} R),
\]
where $t^{\vec{a}} = t_1^{a_1} \cdots t_m^{a_m}$.
\end{lem}

\begin{proof}
In position $\si_i$, we have the $R_{\si_i}$-module 
$L_{K_m} (t^{\vec{a}} R)(\sigma_i) = R_{\sigma_i} \otimes_R t^{\vec{a}} R 
\cong t_i^{a_i} R_{\sigma_i} = [a_i, \infty)_i$. 
By naturality, these pointwise isomorphisms are compatible with the transition maps.
\end{proof}

In the case $m=2$, Figure~\ref{fig:Strips} illustrates the $R$-modules $t_1^a R / t_1^b R$, $t_2^a R / t_2^b R$, and $t_1^{a_1}t_2^{a_2} R$, which can be thought of as the 
objects $[a,b)_1$, $[a,b)_2$ and $[(a_1,a_2),\infty)$ in $\DD{K_2}$, by Lemmas~\ref{lem:LocalizTorsion} and \ref{lem:LocalizFree}. 

\begin{lem}\label{lem:MinDegree}
Let $M$ be a torsion-free finitely generated $R_{\si_i}$-module, for some $1 \leq i \leq m$. Let $x \in M$ be a (non-zero) element of minimal degree in the %
$i$\textsuperscript{th} direction, 
that is:
\[
\abs{x}_i \leq \abs{y}_i \quad \text{for all } y \in M,
\]
and let $\lan x \ran \subseteq M$ denote the $R_{\si_i}$-submodule generated by $x$. Then the quotient map $M \surj M / \lan x \ran$ admits a section. In particular, %
$M / \lan x \ran$ is also torsion-free.
\end{lem}

\begin{proof}
By the equivalence of categories $\fgMod{R_{\si_i}} \cong \fgMod{\kk[t_i]}$, we may assume $m=1$. The \mbox{$\kk[t]$-module} $M$ has an interval decomposition 
$M \cong \bigoplus_j [a_j,\infty)$. 
Writing $a = \min_j \{ a_j \}$, the element $x \in M$ must lie in the submodule
\[
M' = \bigoplus_{\substack{j \\ a_j = a}} [a,\infty)
\]
because of its minimal degree $\abs{x} = a$. 
There is an automorphism of the $\kk$-vector space %
$M'(a) \cong \kk^r$ 
sending $x$ to a generator of a summand $[a,\infty)$ which moreover extends to an automorphism of the $\kk[t]$-module $M'$. Hence we may assume that $x$ is a generator of a summand $[a,\infty)$, in which case the quotient module $M / \lan x \ran$ consists of the other summands. 
\end{proof}

The following lemma follows from \cite{AtiyahM69}*{\S 3}, and will be used in the next proof.

\begin{lem}\label{lem:Localiz}
Let $R$ be a commutative ring and $S \subset R$ a multiplicative set. Let $f \colon M \to N$ be a map of $R$-modules which is an $S$-localization, i.e., isomorphic to $M \to S^{-1} M \cong (S^{-1} R) \ot_{R} M$.
\begin{enumerate}
\item\label{item:SubSource} Given an $R$-submodule $M' \subseteq M$ and taking $N' \dfn f(M) \subseteq N$, the composite
\[
\xymatrix{
M' \ar[r]^-{f} & N' \ar[r] & S^{-1} N' \\
}
\]
is an $S$-localization.
\item\label{item:SubTarget} Given an $S^{-1} R$-submodule $N' \subseteq N$ and taking $M' \dfn f^{-1}(N') \subseteq M$, the map $f' \colon M' \to N'$ obtained by restricting $f$ is an $S$-localization.
\end{enumerate}
\end{lem}

\begin{lem}\label{lem:Scanning}
Let $M$ be a non-zero object in $\cat{F}$. Then there exists a monomorphism $\mu \colon [\vec{a}, \infty) \inj M$ for some $\vec{a} \in \N^m$ such that the quotient module $\coker(\mu) = M / [\vec{a}, \infty)$ is also in $\cat{F}$.
\end{lem}

\begin{proof}
Note that $M$ being non-zero and torsion-free forces $M_{[m]} \neq 0$, which in turn forces $M_{\si_i} \neq 0$ for all $1 \leq i \leq m$. We will construct certain subobjects
\[
M^{(m)} \subseteq M^{(m-1)} \subseteq \cdots \subseteq M^{(1)} \subseteq M^{(0)} = M
\]
in $m$ steps. The process is illustrated in Figure~\ref{fig:Scanning} in the case $m=2$.

\textbf{Step $i=1$.} Since $M_{\si_1} \neq 0$ is a torsion-free $R_{\si_1}$-module, by Proposition~\ref{pr:wesplit}, it admits a finite direct sum decomposition
\begin{equation}\label{eq:Decompo}
M_{\si_1} \cong \bigoplus_j [a_{1,j}, \infty)_1.
\end{equation}
Take $a_1 \dfn \min_j \{ a_{1,j} \}$ and take the submodule $M^{(1)}_{\si_1} \subseteq M_{\si_1}$ corresponding to 
\[
M^{(1)}_{\si_1} \cong \bigoplus_{\substack{j\\a_{1,j=a_1}}} [a_{1,j}, \infty)_1
\]
via the decomposition~\eqref{eq:Decompo}, i.e., the summands with the minimal starting point. Now apply the transition map $\phy_1 \colon M_{\si_1} \to M_{[m]}$ and consider the $R_{[m]}$-submodule
\[
M^{(1)}_{[m]} \dfn t_1^{-1} \phy_1(M^{(1)}_{\si_1}) \cong R_{[m]} \ot_{R_{\si_1}} \phy_1(M^{(1)}_{\si_1}) \subseteq M_{[m]}
\]
with transition map
\[
\phy^{(1)}_1 \colon M^{(1)}_{\si_1} \to M^{(1)}_{[m]}
\]
induced by $\phy_1$. 
Via the decomposition~\eqref{eq:Decompo}, $M^{(1)}_{[m]} \subseteq M_{[m]}$ corresponds to a summand inclusion:
\[
\xymatrix{
M^{(1)}_{[m]} \ar@{-}[d]_{\cong} \ar@{^{(}->}[r] & M_{[m]} \ar@{-}[d]^{\cong} \\
\bigoplus_{a_{1,j=a_1}} R_{[m]} \ar@{^{(}->}[r] & \bigoplus_j R_{[m]}. \\
}
\]
For the remaining positions $\ell \neq 1$, take the pullback
\[
M^{(1)}_{\si_{\ell}} \dfn \phy_{\ell}^{-1} (M^{(1)}_{[m]}) \subseteq M_{\si_{\ell}}
\]
with transition map
\[
\phy^{(1)}_{\ell} \colon M^{(1)}_{\si_{\ell}} \to M^{(1)}_{[m]}
\]
induced by $\phy_{\ell}$. 
The submodules $M^{(1)}_{\si_i} \subseteq M_{\si_i}$ assemble into a subobject $M^{(1)} \subseteq M$. Note that the condition $M^{(1)}_{[m]} \neq 0$ ensures $M^{(1)} \neq 0$.

\textbf{Steps $i > 1$.} Starting from $M^{(1)}$, repeat the process for the step $i=2$, which yields a subobject $M^{(2)} \subseteq M^{(1)}$ satisfying $M^{(2)} \neq 0$. Repeat the process for the steps $i=3, \ldots, m$. Collect the numbers $a_i$ into a multidegree $\vec{a} = (a_1, \ldots, a_m) \in \N^m$.

\textbf{Constructing a monomorphism $\la \colon [\vec{a},\infty) \inj M^{(m)}$.} Note that a morphism $[\vec{a},\infty) \to N$ in $\DD{K_m}$ is the same data as a family of elements $x_i \in N_{\si_i}(\vec{a})$ %
that all map to the same element in the terminal position:
\[
\phy_i(x_i) = \phy_j(x_j) \in N_{[m]}(\vec{a}) \quad \text{for all } 1 \leq i, j, \leq m.
\]
We will produce such elements for $N = M^{(m)}$. By construction, the $R_{\si_m}$-module $M^{(m)}_{\si_m}$ has an interval decomposition
\[
M^{(m)}_{\si_m} \cong \bigoplus_j [a_{m}, \infty)_m.
\]
Pick one of the summands $[a_{m}, \infty)_m \inj M^{(m)}_{\si_m}$ and let $x_m \in M^{(m)}_{\si_m}(\vec{a})$ be the image of the generator $t^{\vec{a}} 1 \in [a_{m}, \infty)_m (\vec{a}) \cong \kk$. 
Take $x := \phy^{(m)}_m(x_m) \in M^{(m)}_{[m]}(\vec{a})$. To find the next element $x_{m-1}$, consider the pullback diagram in \mbox{$R_{\si_{m-1}}$-modules}
\[
\xymatrix @R-0.2pc {
\therefore \exists \, x_{m-1} \ar@{}[r]|{\in} & M^{(m)}_{\si_{m-1}} \pb \ar[d]_{\inc_{\si_{m-1}}} \ar[r]^-{\phy^{(m)}_{m-1}} & M^{(m)}_{[m]} \ar[d]^{\inc_{[m]}} & x \ar@{}[l]|{\ni} \\
\exists \, \tild{x}_{m-1} \ar@{}[r]|{\in} & M^{(m-1)}_{\si_{m-1}} \ar@{-}[d]_{\cong} \ar[r]^-{\phy^{(m-1)}_{m-1}} & M^{(m-1)}_{[m]} \ar@{-}[d]^{\cong} & x \ar@{}[l]|{\ni} \\
& \bigoplus [a_{m-1}, \infty)_{m-1} \ar[r]^-{\bigoplus \inc} & \bigoplus R_{[m]} \\
}
\]
where the rightward maps invert $t_{m-1}$. Since the bottom map is an isomorphism of $\kk$-modules in multidegree $\vec{a}$, there is an element $\tild{x}_{m-1} \in  M^{(m-1)}_{\si_{m-1}}$ satisfying $\phy^{(m-1)}_{m-1}(\tild{x}_{m-1}) = x$, hence also an element in the pullback $x_{m-1} \in  M^{(m)}_{\si_{m-1}}$ satisfying $\phy^{(m)}_{m-1}(x_{m-1}) = x$.

To find the next element $x_{m-2}$, apply the same argument to the pullback diagram in $R_{\si_{m-2}}$-modules
\[
\xymatrix @R-0.2pc {
\therefore \exists \, x_{m-2} \ar@{}[r]|{\in} & M^{(m)}_{\si_{m-2}} \pb \ar[d]_{\inc_{\si_{m-2}}} \ar[r]^-{\phy^{(m)}_{m-2}} & M^{(m)}_{[m]} \ar[d]^{\inc_{[m]}} & x \ar@{}[l]|{\ni} \\
& M^{(m-1)}_{\si_{m-2}} \pb \ar[d]_{\inc_{\si_{m-2}}} \ar[r]^-{\phy^{(m-1)}_{m-2}} & M^{(m-1)}_{[m]} \ar[d]^{\inc_{[m]}} & \\
\exists \, \tild{x}_{m-2} \ar@{}[r]|{\in} & M^{(m-2)}_{\si_{m-2}} \ar@{-}[d]_{\cong} \ar[r]^-{\phy^{(m-2)}_{m-2}} & M^{(m-2)}_{[m]} \ar@{-}[d]^{\cong} & x \ar@{}[l]|{\ni} \\
& \bigoplus [a_{m-2}, \infty)_{m-2} \ar[r]^-{\bigoplus \inc} & \bigoplus R_{[m]}. \\
}
\]
Repeat the process to find the remaining elements $x_i \in M^{(m)}_{\si_i}(\vec{a})$, which jointly define a morphism $\la \colon [\vec{a},\infty) \to M^{(m)}$ in $\DD{K_m}$. For each $1 \leq i \leq m$, the $R_{\si_i}$-module $M^{(m)}_{\si_i}$ is torsion-free, hence the map
\[
\la_{\si_i} \colon [a_i ,\infty)_{i} \inj M^{(m)}_{\si_i}
\]
is a monomorphism. By Lemma~\ref{lem:PointwiseExact}, the map $\la \colon [\vec{a},\infty) \inj M^{(m)}$ is also a monomorphism. 
Define $\mu \colon [\vec{a}, \infty) \inj M$ as %
$\la$ followed by the inclusion $M^{(m)} \subseteq M$.

\textbf{The cokernel of $\mu$ is torsion-free.} For each $1 \leq i \leq m$, let us show that the $R_{\si_i}$-module
\[
\coker(\mu)_{\si_i} = M_{\si_i} / \lan x_i \ran
\]
is torsion-free. Let $z \in M_{\si_i} \setminus \lan x_i \ran$. We want to show that for any non-zero scalar $\al \in R_{\si_i}$, the condition $\al z \in M_{\si_i} \setminus \lan x_i \ran$ still holds. We distinguish two cases, depending where $z$ lies in the decreasing filtration
\[
M_{\si_i} = M^{(0)}_{\si_i} \supseteq M^{(1)}_{\si_i} \supseteq \cdots \supseteq M^{(m)}_{\si_i}.
\]

\textbf{Case: $z$ does not survive all the way.} Assume $z \in M^{(\ell-1)}_{\si_i} \setminus M^{(\ell)}_{\si_i}$ for some $1 \leq \ell \leq m$. Then the condition $\al z \in M^{(\ell-1)}_{\si_i} \setminus M^{(\ell)}_{\si_i}$ also holds. In the case $i \neq \ell$, this follows from the pullback square on the left:
\[
\xymatrix @R-0.2pc @C+0.6pc {
& \bigoplus R_{[m]} \ar@{-}[d]_{\cong} & \bigoplus [a_{\ell}, \infty)_{\ell} \ar[l]_-{\bigoplus \inc} \ar@{-}[d]^{\cong} \\
M^{(\ell)}_{\si_i} \pb \ar[d]_{\inc_{\si_i}} \ar[r]^-{\phy^{(\ell)}_i} & M^{(\ell)}_{[m]} \ar[d]^{\inc_{[m]}} & M^{(\ell)}_{\si_{\ell}} \ar[d]^{\inc_{\si_{\ell}}} \ar[l]_-{\phy^{(\ell)}_{\ell}} \\
M^{(\ell-1)}_{\si_i} \ar[r]^-{\phy^{(\ell-1)}_i} & M^{(\ell-1)}_{[m]} & M^{(\ell-1)}_{\si_{\ell}} \ar[l]_-{\phy^{(\ell-1)}_{\ell}} \\
& \bigoplus_j R_{[m]} \ar@{-}[u]^{\cong} & \bigoplus_j [a_{\ell,j}, \infty)_{\ell} \ar[l]_-{\bigoplus \inc} \ar@{-}[u]_{\cong} \\
}
\]
and the fact that $R_{[m]}$ has no zero-divisors. In the case $i = \ell$, it follows from the rightmost column and the fact that each $R_{\si_{\ell}}$-module $[a_{\ell,j}, \infty)_{\ell}$ is torsion-free. In either case, $\al z$ cannot lie in the submodule $\lan x_i \ran \subseteq M^{(m)}_{\si_i} \subseteq M^{(\ell)}_{\si_i}$.

\textbf{Case: $z$ survives all the way.} Assume $z \in M^{(m)}_{\si_i}$. By construction, the $R_{\si_i}$-module $M^{(m)}_{\si_i}$ is concentrated in degrees $\vec{d} \in \si_i^{-1} \N^m$ with $d_i \geq a_i$. Since $x_i$ achieves the minimal degree $\abs{x_i}_i = a_i$, the quotient $M^{(m)}_{\si_i} / \lan x_i \ran$ is torsion-free, by Lemma~\ref{lem:MinDegree}.
\end{proof}

\begin{figure}[H]
\begin{adjustbox}{width=\textwidth}
\begin{tikzpicture}[scale=1]
 \begin{axis}[
axis lines=middle, 
ticks=none, 
axis equal, 
xmin=-0.5, xmax=5, ymin=-1, ymax=5]
\draw [blue, fill=blue] (1,2) circle (\VertexSize);
\draw [blue, thick] (1,6) -- (1,2) -- (6,2);
\fill [blue, opacity=0.1] (1,2) rectangle (6,6);
\draw [blue, dotted] (1,2) -- (1,0);
\node [below] at (1,0) {$a_1$};
\draw [green, fill=green] (1,4) circle (\VertexSize);
\draw [green, thick] (1,6) -- (1,4) -- (6,4);
\fill [green, opacity=0.1] (1,4) rectangle (6,6);
\draw [red, fill=red] (3,1) circle (\VertexSize);
\draw [red, thick] (3,6) -- (3,1) -- (6,1);
\fill [red, opacity=0.1] (3,1) rectangle (6,6);
\node at (3,-0.8) {$M = M^{(0)}$};
 \end{axis}
\end{tikzpicture}
\hspace{1ex}
\begin{tikzpicture}[scale=1]
 \begin{axis}[
axis lines=middle, 
ticks=none, 
axis equal, 
xmin=-0.5, xmax=5, ymin=-1, ymax=5]
\draw [blue, fill=blue] (1,2) circle (\VertexSize);
\draw [blue, thick] (1,6) -- (1,2) -- (6,2);
\fill [blue, opacity=0.1] (1,2) rectangle (6,6);
\draw [blue, dotted] (1,2) -- (0,2);
\node [left] at (0,2) {$a_2$};
\draw [green, fill=green] (1,4) circle (\VertexSize);
\draw [green, thick] (1,6) -- (1,4) -- (6,4);
\fill [green, opacity=0.1] (1,4) rectangle (6,6);
\node at (2.7,-0.8) {$M^{(1)}$};
 \end{axis}
\end{tikzpicture}
\hspace{1ex}
\begin{tikzpicture}[scale=1]
 \begin{axis}[
axis lines=middle, 
ticks=none, 
axis equal, 
xmin=-0.5, xmax=5, ymin=-1, ymax=5]
\draw [blue, fill=blue] (1,2) circle (\VertexSize);
\draw [blue, thick] (1,6) -- (1,2) -- (6,2);
\fill [blue, opacity=0.1] (1,2) rectangle (6,6);
\draw [blue, dotted] (1,2) -- (1,0);
\node [below] at (1,0) {$a_1$};
\draw [blue, dotted] (1,2) -- (0,2);
\node [left] at (0,2) {$a_2$};
\node at (2.7,-0.8) {$M^{(2)}$};
 \end{axis}
\end{tikzpicture}
\end{adjustbox}
\caption{The ``scanning'' process in the proof of Lemma~\ref{lem:Scanning}.}
\label{fig:Scanning}
\end{figure}

\begin{lem}\label{lem:Dim1}
Let $M$ be an object in $\cat{F}$.
\begin{enumerate}
\item\label{item:Dim0} If %
$\rank M = 0$ holds, 
then $M=0$ holds.
\item\label{item:Dim1} If %
$\rank M = 1$ holds, 
then $M$ is isomorphic to $[\vec{d}, \infty)$ for some $\vec{d}$.
\end{enumerate}
\end{lem}

\begin{proof}
\Eqref{item:Dim0} 
Since each $R_{\si_i}$-module $M_{\si_i}$ is torsion-free, it embeds in its localization $R_{[m]} \ot_{R_{\si_i}} M_{\si_i} \cong M_{[m]} = 0$, yielding $M_{\si_i} = 0$ and thus $M = 0$.

\Eqref{item:Dim1} By Lemma~\ref{lem:Scanning}, there exists a monomorphism $\mu \colon [\vec{d}, \infty) \inj M$ such that the quotient $\coker(\mu)$ is also in $\cat{F}$. Evaluating at the position $[m]$ yields a short exact sequence of $R_{[m]}$-modules
\[
\xymatrix{
0 \ar[r] & R_{[m]} \ar[r]^-{\mu_{[m]}} & M_{[m]} \ar[r] & \coker(\mu)_{[m]} \ar[r] & 0 \\
}
\]
from which we obtain the dimension count
\[
\dim_{\kk} \coker(\mu)_{[m]} = \dim_{\kk} M_{[m]} - \dim_{\kk} R_{[m]} = 1 - 1 = 0.
\]
Part~\Eqref{item:Dim0} then implies $\coker(\mu) = 0$, so that $\mu \colon [\vec{d}, \infty) \ral{\cong} M$ is an isomorphism.
\end{proof}

\begin{lem}\label{lem:IndecompoFree}
For any %
$\vec{d} \in \N^m$, the object $[\vec{d},\infty)$ in $\DD{K_m}$ is indecomposable.
\end{lem}

\begin{proof}
Consider a direct sum decomposition $[\vec{d},\infty) \cong P \op Q$. Since the subcategory $\cat{F}$ is closed under subobjects, both summands $P$ and $Q$ lie in $\cat{F}$. A dimension count as in the proof of Lemma~\ref{lem:Dim1} \Eqref{item:Dim1} then forces $P = 0$ or $Q = 0$.
\end{proof}

\begin{lem}\label{lem:Battleship}
Let $M \in \cat{F}$ be a torsion-free module of rank $r = \dim_{\kk} M_{[m]} \geq 2$. Then applying the ``scanning'' Lemma~\ref{lem:Scanning} $r-1$ times produces an epimorphism %
of the form $M \surj [\vec{a}, \infty)$.
\end{lem}

The process is illustrated in Figure~\ref{fig:Battleship}.

\begin{proof}
Applying Lemma~\ref{lem:Scanning} $r-1$ times to $M$ produces an epimorphism from $M$ to a torsion-free module of rank $r-(r-1)=1$, hence isomorphic to $[\vec{a}, \infty)$ for some $\vec{a} \in \N^m$, by Lemma~\ref{lem:Dim1}.
\end{proof}

\begin{figure}[H]
\begin{adjustbox}{width=\textwidth}
\begin{tikzpicture}[scale=1]
 \begin{axis}[
axis lines=middle, 
ticks=none, 
axis equal, 
xmin=-0.2, xmax=5, ymin=-1, ymax=5]
\draw [blue, fill=blue] (1,2) circle (\VertexSize);
\draw [blue, thick] (1,6) -- (1,2) -- (6,2);
\fill [blue, opacity=0.1] (1,2) rectangle (6,6);
\draw [blue, dotted] (1,2) -- (1,0);
\node [below] at (1,0) {$d^1_1$};
\draw [blue, dotted] (1,2) -- (0,2);
\node [left] at (0,2) {$d^1_2$};
\node [red] at (1,2) {$\bigtimes$};
\draw [green, fill=green] (1,4) circle (\VertexSize);
\draw [green, thick] (1,6) -- (1,4) -- (6,4);
\fill [green, opacity=0.1] (1,4) rectangle (6,6);
\draw [red, fill=red] (3,1) circle (\VertexSize);
\draw [red, thick] (3,6) -- (3,1) -- (6,1);
\fill [red, opacity=0.1] (3,1) rectangle (6,6);
\node at (3,-0.8) {$M = M^{0}$};
 \end{axis}
\draw [->>] (6.9,3) -- (7.4,3);
\end{tikzpicture}
\begin{tikzpicture}[scale=1]
 \begin{axis}[
axis lines=middle, 
ticks=none, 
axis equal, 
xmin=-0.2, xmax=5, ymin=-1, ymax=5]
\draw [green, fill=green] (1,4) circle (\VertexSize);
\draw [green, thick] (1,6) -- (1,4) -- (6,4);
\fill [green, opacity=0.1] (1,4) rectangle (6,6);
\draw [green, dotted] (1,4) -- (1,0);
\node [below] at (1,0) {$d^2_1$};
\draw [green, dotted] (1,4) -- (0,4);
\node [left] at (0,4) {$d^2_2$};
\node [red] at (1,4) {$\bigtimes$};
\draw [red, fill=red] (3,1) circle (\VertexSize);
\draw [red, thick] (3,6) -- (3,1) -- (6,1);
\fill [red, opacity=0.1] (3,1) rectangle (6,6);
\node at (2.7,-0.8) {$M^{1}$};
 \end{axis}
\draw [->>] (6.9,3) -- (7.4,3);
\end{tikzpicture}
\begin{tikzpicture}[scale=1]
 \begin{axis}[
axis lines=middle, 
ticks=none, 
axis equal, 
xmin=-0.2, xmax=5, ymin=-1, ymax=5]
\draw [red, fill=red] (3,1) circle (\VertexSize);
\draw [red, thick] (3,6) -- (3,1) -- (6,1);
\fill [red, opacity=0.1] (3,1) rectangle (6,6);
\draw [red, dotted] (3,1) -- (3,0);
\node [below] at (3,0) {$a_1$};
\draw [red, dotted] (3,1) -- (0,1);
\node [left] at (0,1) {$a_2$};
\node at (2.7,-0.8) {$M^{2} \cong [\vec{a},\infty)$};
 \end{axis}
\end{tikzpicture}
\end{adjustbox}
\caption{The ``battleship game'' in the proof of Lemma~\ref{lem:Battleship}.}
\label{fig:Battleship}
\end{figure}

\begin{lem}\label{lem:LiftingM2}
In the case $m=2$, let $M$ be a torsion-free module in $\cat{F}$. Then any epimorphism $q \colon M \surj [\vec{a}, \infty)$ constructed as in Lemma~\ref{lem:Battleship} is split.
\end{lem}

\begin{proof}
Write $y_i := t^{\vec{a} - a_i \vec{e}_i} t_i^{a_i} 1 \in [a_i, \infty)_{i}$ for the canonical generator, shifted to bidegree $\vec{a}$ for convenience. Those agree in the terminal corner:
\[
\phy_1(y_1) = \phy_2(y_2) = y = t^{\vec{a}} 1 \in \kk[t_1^{\pm}, t_2^{\pm}] = R_{[2]} = [\vec{a}, \infty)_{[2]}.
\]
A section of $q$ amounts to preimages $x_i \in M_{\si_i}(\vec{a})$ of the generators $y_i$ that agree in the terminal corner: $\phy_1(x_1) = \phy_2(x_2) \in M_{[2]}(\vec{a})$, as illustrated in the diagram
\[
\xymatrix @R-0.5pc @C-0.5pc {
& y_1 \in [a_1,\infty)_1 \ar@{^{(}->}[rr]^-{\phy_1} & & R_{[2]} \ni y \\
x_1 \in M_{\si_1} \ar@{->>}[ur]^{q_{\si_1}} \ar@{^{(}->}[rr]^-{\phy_1} & & M_{[2]} \ar@{->>}[ur]^{q_{[2]}} \ni x & \\
& & & [a_2, \infty)_2 \ar@{^{(}->}[uu]_{\phy_2} \ni y_2 \\
& & M_{\si_2} \ni x_2. \ar@{->>}[ur]^{q_{\si_2}} \ar@{^{(}->}[uu]_{\phy_2} & \\
}
\]
In position $\si_2 = \{1\}$, consider the epimorphism of $R_{\si_2}$-modules %
$q_{\si_2} \colon M_{\si_2} \surj [a_2, \infty)_{2}$ and pick any preimage $x_2$ of $y_2$. Take $x := \phy_2(x_2) \in M_{[2]}(\vec{a})$, which is a preimage of $y \in [\vec{a}, \infty)_{[2]}(\vec{a})$. 
Now the $R_{\si_1}$-module $M_{\si_1}$ admits an interval decomposition
\[
M_{\si_1} \cong \bigoplus_j [a^j, \infty)_1.
\]
By construction of the epimorphism in Lemma~\ref{lem:Battleship}, we have $a_1 \geq a^j_1$ for all $j$. Hence the localization map 
$\phy_1 \colon M_{\si_1} \inj M_{[2]}$ 
is an isomorphism in bidegree $\vec{a}$. Therefore there exists a unique $x_1 \in M_{\si_1}(\vec{a})$ satisfying $\phy_1(x_1) = x$. %
One readily checks that 
$x_1$ satisfies
\[
\phy_1 q_{\si_1}(x_1) = \phy_1 (y_1).
\]
Since the localization map $\phy_1 \colon [a_1, \infty)_{1} \inj R_{[2]}$ %
is injective, this forces $q_{\si_1}(x_1) = y_1$, so that $x_1$ is a preimage of $y_1$.
\end{proof}

\begin{prop}\label{pr:M2Decompo}
In the case $m=2$, each torsion-free object $M \in \cat{F}$ decomposes as a direct sum of modules of the form $[\vec{d}, \infty )$. 
\end{prop}

\begin{proof}
Write %
$r = \rank M$. 
The case $r \leq 1$ of the statement was proved in Lemma~\ref{lem:Dim1}, so we assume $r > 1$. 
Pick any epimorphism $M \surj [\vec{a}, \infty)$ constructed as in Lemma~\ref{lem:Battleship}, which admits a section, by Lemma~\ref{lem:LiftingM2}. 
This produces a splitting $M \cong P \op [\vec{a}, \infty)$, so that the complementary summand $P$ satisfies
\[
\rank P = \rank M - 1 = r-1.
\]
By induction on $r$, we obtain a direct sum decomposition of $M$ into %
$r$ summands of the form $[\vec{d}, \infty)$.
\end{proof}

\section{Torsion-free objects in dimension \texorpdfstring{$m \geq 3$}{at least 3}}

In this section, we show that the analogue of Proposition~\ref{pr:M2Decompo} \emph{fails} for $m \geq 3$. 
In Section~\ref{sec:Wild}, we will show that it fails miserably.

The following lemma says that the algorithm used in Proposition~\ref{pr:M2Decompo} has a detection property: \emph{If} a torsion-free object admits a direct sum decomposition into %
objects of the form 
$[\vec{d},\infty)$, then the algorithm is guaranteed to find such a decomposition.

\begin{lem}\label{lem:Detect}
Let $M$ be a torsion-free object in $\cat{F}$ of the form
\[
M \cong \bigoplus_{j} [\vec{a}^j, \infty),
\]
and write $\vec{b} := \min_{j} \vec{a}^j$, the minimal multidegree in \emph{lexicographic} order.
\begin{enumerate}
\item \label{item:MinDegree} The subobject $M^{(m)}$ constructed in Lemma~\ref{lem:Scanning} is of the form
\[
M^{(m)} \cong \bigoplus_{\substack{j\\\vec{a}^j = \vec{b}}} [\vec{b}, \infty)
\]
i.e., keeping only the summands that start in the minimal multidegree $\vec{b}$. 
\item \label{item:ScanningSplit} Any monomorphism $\mu \colon [\vec{b}, \infty) \inj M$ constructed in Lemma~\ref{lem:Scanning} is split.
\item \label{item:BattleshipSplit} Any epimorphism $M \surj [\vec{a}, \infty)$ constructed as in Lemma~\ref{lem:Battleship} is split.
\end{enumerate}
\end{lem}

\begin{proof}
\Eqref{item:MinDegree} Since isomorphisms in $\DD{K_m}$ preserve the multigrading, we may assume $M = \bigoplus_{j} [\vec{a}^j, \infty)$. The successive subobjects constructed in Lemma~\ref{lem:Scanning} consist of fewer and fewer of the summands:
\[
M^{(\ell)} = \bigoplus_{\substack{j\\ \vec{a}^j_i = \vec{b}_i \text{ for } 1 \leq i \leq \ell}} [\vec{a}^j, \infty).
\]
\Eqref{item:ScanningSplit} By the analogue of Lemma~\ref{lem:MinDegree} in $\DD{K_m}$ instead of $R_{\si_i}$-modules, the monomorphism $\la \colon [\vec{b}, \infty) \inj M^{(m)}$ is isomorphic to the inclusion of one of the summands. The inclusion $M^{(m)} \subseteq M$ is also an inclusion of summands. Hence the composite monomorphism $\mu \colon [\vec{b}, \infty) \inj M$ is split, and its cokernel $\coker(\mu)$ still satisfies the assumption of part~\Eqref{item:MinDegree}. The epimorphism constructed in Lemma~\ref{lem:Battleship} is then a composite of $r-1$ split epimorphisms, which proves \Eqref{item:BattleshipSplit}.
\end{proof}

\begin{ex}\label{ex:NotSplit}
For $M$ in $\cat{F}$ and $\vec{a} \in \N^m$, an epimorphism $f \colon M \surj [\vec{a},\infty)$ need \emph{not} admit a section. Here is an example in the case $m=2$. Consider the map of $R$-modules
\[
\xymatrix @C+1.4pc {
t_1 R \op t_2 R \ar[r]^-{g = \smat{\inc & \inc}} & R \\
}
\]
and denote the generators respectively by
\[
\begin{cases}
x = t_1 \cdot 1 \in t_1 R, &\abs{x} = (1,0) \\
y = t_2 \cdot 1 \in t_2 R, &\abs{y} = (0,1) \\
z = 1 \in R, &\abs{z} = (0,0). \\
\end{cases}
\]
Take $f = L_K(g)$, which can be written using Notation~\ref{nota:Quadrant}:
\[
\xymatrix @C+1.4pc {
M = [(1,0),\infty) \op [(0,1),\infty) \ar[r]^-{f = \smat{\inc & \inc}} & [(0,0),\infty). \ar@/^1.5pc/@{-->}[l]^{\not\exists \, s} \\
}
\]
The map $f$ is an epimorphism in $\DD{K_2}$, since in position $\si_1$ it is the epimorphism of $R_{\si_1}$-modules
\[
\xymatrix @C+1.4pc {
[1,\infty)_1 \op [0,\infty)_1 \ar@{->>}[r]^-{f_{\si_1} = \smat{\inc & \id}} & [0,\infty)_1 \\
}
\]
and similarly for position $\si_2$. However, $f$ does \emph{not} admit a section $s$. Indeed, %
a section of \mbox{$R_{\si_1}$-modules} $s_1$ in position $\si_1$ and a section $s_2$ in position $\si_2$ cannot assemble into a compatible section $s$ in $\DD{K_2}$.
\end{ex}

\begin{lem}\label{lem:M3Indecompo}
For $m = 3$, there exists a torsion-free object in $\cat{F}$ of rank $2$ that is indecomposable. %
\end{lem}

\begin{proof}
Consider the $R$-module
\[
\ol{M} = R \lan w, y, z \mid t_1 w = t_3 y - t_2 z \ran, \quad \abs{w} = (0,1,1), \: \abs{y} = (1,1,0), \: \abs{z} = (1,0,1) 
\]
and take its localization $M \dfn L_{K_3}(\ol{M})$. The ``scanning'' Lemma~\ref{lem:Scanning} produces a monomorphism
\[
\mu \colon [(0,1,1), \infty) \inj M
\]
picking out the image of $w$. More precisely, consider the short exact sequence of $R$-modules
\[
\xymatrix{
0 \ar[r] & t^{\abs{w}} R \ar@{^{(}->}[r]^-{w} & \ol{M} \ar@{->>}[r] & \ol{M} / \lan w \ran \ar[r] & 0. \\
}
\]
Applying the exact functor $L_{K_3}$ yields a short exact sequence in $\DD{K_3}$
\[
\xymatrix{
0 \ar[r] & [(0,1,1), \infty) \ar@{^{(}->}[r]^-{\mu} & M \ar@{->>}[r]^-{q} & L_{K_3}(\ol{M} / \lan w \ran) \ar[r] & 0. \\
}
\]
Restricting to the positions $\si \subseteq [3]$ containing $1$ and evaluating the multidegree $\vec{d}$ at $d_1 = 0$ defines functors
\[
\xymatrix @C+0.5pc {
\DD{K_3} \ar[r]^-{\text{restrict}} & \DD{K_3}[t_1^{\pm}] \ar[r]^-{\text{set } d_1=0}_-{\cong} & \DD{K_2} \\ 
}
\]
whose composite sends the epimorphism $q \colon M \surj \coker(\mu)$ to the non-split epimorphism from Example~\ref{ex:NotSplit}. Hence the epimorphism $q$ itself is not split. By Lemma~\ref{lem:Detect}, $M$ does not decompose as a direct sum of modules of the form $[\vec{d}, \infty )$. By Lemma~\ref{lem:Dim1}, any direct sum decomposition of $M$ would be of that form.
\end{proof}

\begin{prop}\label{pr:Mgeq3}
For any $m \geq 3$, there exists a torsion-free module in $\cat{F}$ of rank $2$ that is indecomposable. 
\end{prop}

\begin{proof}
Here we will denote subsets of $[3]$ by the letter $\tau$, and the graded ring $S = \kk[t_1, t_2, t_3]$ to avoid ambiguity with $[m]$. 
Consider the functor $G \colon \DD{K_3} \to \DD{K_m}$ %
defined by
\[
G(M)_{\si_i} = \begin{cases}
R_{\si_i} \ot_{S_{\tau_i}} M_{\tau_i} &\text{if } 1 \leq i \leq 3 \\
R_{\si_i} \ot_{S_{[3]}} M_{[3]} &\text{if } 4 <  i \leq m \\
\end{cases}
\]
with the canonical localization maps $\phy_i \colon G(M)_{\si_i} \to G(M)_{[m]}$. The composite of $G$ followed by the functors
\[
\xymatrix @C+3.5pc {
\DD{K_m} \ar[r]^-{\text{restrict}} & \DD{K_m}[t_4^{\pm}, \cdots, t_m^{\pm}] \ar[r]^-{\text{set } d_4=0, \ldots, d_m=0}_-{\cong} & \DD{K_3} \\ 
}
\]
is an equivalence of categories. Now take $M$ to be the rank $2$ indecomposable torsion-free object in $\DD{K_3}$ from Lemma~\ref{lem:M3Indecompo}. Then $G(M)$ is a rank $2$ torsion-free object in $\DD{K_m}$ which is indecomposable.
\end{proof}

The functors used in the proof allow us to compare the categories $\DD{K}$ for different $K$. The next lemma generalizes the functor $G$. Since we will have different vertex sets $V$, we will include the vertex set in the notation $\DD{K;V}$.

\begin{lem}\label{lem:ExtendComplex}
Let $K$ be a simplicial complex on a vertex set $V$ and $L$ a simplicial complex on the vertex set $W$, with $V \subseteq W$. Assume that $L$ extends $K$, that is:
\[
K \subseteq L \vert_{V} \dfn L \cap \PP{V}.
\]
Then the following construction defines an exact functor
\[
\Phi_{K,L} \colon \DD{K;V} \to \DD{L;W}.
\]
Given %
$M$ in $\DD{K;V}$, the module $\Phi_{K,L}(M)$ in $\DD{L;W}$ has at position $\si \in L^c$
\[
\Phi_{K,L}(M)_{\si} \dfn R_{\si} \ot_{R_{\si \cap V}} M_{\si \cap V},
\]
which is defined since $\si \cap V \in K^c$ is a missing face of $K$. For $\si \subseteq \tau$, the transition map $\Phi_{K,L}(M)_{\si} \to \Phi_{K,L}(M)_{\tau}$ is the composite
\[
\xymatrix @C+3.2pc {
R_{\si} \ot_{R_{\si \cap V}} M_{\si \cap V} \ar[r]^-{\mathrm{localize} \ot \phy_{\si \cap V, \tau \cap V}} & R_{\tau} \ot_{R_{\si \cap V}} M_{\tau \cap V} \ar[r] & R_{\tau} \ot_{R_{\tau \cap V}} M_{\tau \cap V}.
}
\]
\end{lem}

\begin{proof}
Exactness of $\Phi_{K,L}$ follows from Lemma~\ref{lem:PointwiseExact} and the flatness of $R_{\si}$ as a module over $R_{\si \cap V}$. 
\end{proof}

\begin{ex}
The functor $G \colon \DD{K_3} \to \DD{K_m}$ in the proof of Proposition~\ref{pr:Mgeq3} was the functor $\Phi_{K_3, K_m}$, over the vertex sets $[3] \subseteq [m]$. The functor $\Phi_{K_1, K_2} \colon \DD{K_1} \to \DD{K_2}$ has the following effect on objects:
\[
\xymatrix{
M_{\si_1} \ar[r] & M_{[1]} & \ar@{|->}[r] & & \kk[t_1, t_2^{\pm}] \ot_{\kk[t_1]} M_{\si_1} \ar[r] & \kk[t_1^{\pm}, t_2^{\pm}] \ot_{\kk[t_1^{\pm}]} M_{[1]} \\
& & & & & \kk[t_1^{\pm}, t_2] \ot_{\kk[t_1^{\pm}]} M_{[1]}. \ar[u].
}
\]
\end{ex}

\section{Wild representation type}\label{sec:Wild}

We first review some background about wild representation type. Some references assume that the base field $\kk$ is algebraically closed, which we do \emph{not} assume here. %
We follow the terminology in \cite{SimsonS07v3}*{\S XIX.1}. %

\begin{nota}
Let $\kk\langle t_1, t_2 \rangle$ denote the $\kk$-algebra of polynomials in two non-commuting variables, i.e., the free (associative unital) $\kk$-algebra on two generators, i.e., the tensor algebra $T_{\kk}(\kk^2)$.

For a $\kk$-algebra $A$, let $\fdMod{A}$ denote the full subcategory of $\Mod{A}$ consisting of modules that are finite-dimensional over $\kk$. Note that if $A$ is itself finite-dimensional, an $A$-module is finite-dimensional if and only if it is finitely generated over $A$, i.e., $\fdMod{A} = \fgMod{A}$.
\end{nota}

\begin{defn}\label{def:Wild}
Let $\cat{A}$ and $\cat{B}$ be %
$\kk$-linear abelian categories.
\begin{enumerate}
\item A functor $T \colon \cat{A} \to \cat{B}$ is a \Def{representation embedding} if it is $\kk$-linear, exact, injective on isomorphism classes (i.e., $T(X) \cong T(Y) \implies X \cong Y$), and sends indecomposables to indecomposables.
\item The category $\cat{A}$ has \Def{wild representation type} if 
there exists a representation embedding $T \colon \fdMod{\kk\langle t_1, t_2 \rangle} \to \cat{A}$.
\item A (finite-dimensional) $\kk$-algebra $A$ has \Def{wild representation type} if its category of finitely generated modules $\fgMod{A}$ does.
\end{enumerate}
\end{defn}

Since we do \emph{not} assume that the base field $\kk$ is algebraically closed, we should say \emph{$\kk$-wild}, but we say wild for short. 
For different notions of wildness, see the discussion after \cite{SimsonS07v3}*{\S XIX Corollary~1.15} and the helpful overviews in \cite{Simson05}*{\S 2} and \cite{ArnoldS05}*{\S 2}.

\begin{lem}\label{lem:Embedding}
For $\kk$-linear abelian 
categories $\cat{A}$ and $\cat{B}$, %
a fully faithful exact $\kk$-linear functor $T \colon \cat{A} \to \cat{B}$ is a representation embedding \cite{SimsonS07v3}*{\S XIX Lemma~1.2}.
\end{lem}

\begin{lem}
A $\kk$-linear abelian category $\cat{A}$ %
has wild representation type if and only if for each finite-dimensional $\kk$-algebra $B$, there exists a representation embedding $T \colon \fgMod{B} \to \cat{A}$ \cite{ArnoldS05}*{Lemma~2.2}.
\end{lem}

\begin{lem}\label{lem:PolyWild}
The category $\fgMod{\kk[t_1,t_2]}$ of finitely generated bigraded modules over the bigraded polynomial algebra $\kk[t_1,t_2]$ has wild representation type.
\end{lem}

\begin{proof}
The poset $\N^2$ contains a rectangular grid $[m] \x [n]$, which has wild representation type for $m,n \geq 5$ \cite{BotnanL23}*{\S 8.2} \cite{Nazarova75}. By Lemma~\ref{lem:Embedding}, the result follows.
\end{proof}

\begin{rem}\label{rem:Ungraded}
The analogue of Lemma~\ref{lem:PolyWild} for the ungraded polynomial algebra also holds \cite{SimsonS07v3}*{\S XIX Theorem~1.11}. The proof therein works over an arbitrary field $\kk$.
\end{rem}

\begin{prop}\label{pr:WildRetract}
Let $K$ be a simplicial complex on the vertex set $[m]$ satisfying the %
proper inclusion $K \subset K_m = \sk_{m-3} \De^{m-1}$. Then the category $\DD{K}$ contains $\fgMod{\kk[s,t]}$ as an exact retract 
(up to equivalence). In particular, $\DD{K}$ has wild representation type.
\end{prop}

\begin{proof}
The condition $K \subset K_m = \sk_{m-3} \De^{m-1}$ is equivalent to: there is a missing face $\si \in K^c$ that does \emph{not} contain two vertices, say, $k, \ell \notin \si$. At position $\si$, the variables $t_k$ and $t_{\ell}$ have not been inverted in $R_{\si}$.

Consider the simplicial complex $\emptyset = \{ \}$ on the vertex set $\{ k, \ell \} \subseteq [m]$. The condition $\emptyset \subseteq K \vert_{\{ k, \ell \}}$ holds (vacuously), so that Lemma~\ref{lem:ExtendComplex} provides an exact functor
\[
\Phi_{\emptyset,K} \colon \DD{\emptyset; \{k, \ell\} } \to \DD{K}.
\]
On the other hand, inverting all the variables $t_i$ for $i \neq k,\ell$ yields a functor
\begin{equation}\label{eq:Retraction}
\xymatrix @C+3.5pc {
\DD{K} \ar[r]^-{\text{restrict}} & \DD{K}[t_i^{\pm} \mid i \neq k, \ell] \ar[r]^-{\text{set } d_i=0 \text{ for } i \neq k,\ell}_-{\cong} & \DD{K\vert_{ \{k, \ell\} }} = \DD{\emptyset; \{k, \ell\} }. \\ 
}
\end{equation}
Here we used the assumption $k,\ell \notin \si$ to obtain the missing faces of the restricted complex:
\begin{align*}
&\si \cap \{ k,\ell \} = \emptyset \in (K^c) \vert_{ \{k, \ell\} } = (K \vert_{ \{k, \ell\} })^c \\
\implies &(K \vert_{ \{k, \ell\} })^c = \PP{\{k, \ell\}}.
\end{align*}
The composite of $\Phi_{\emptyset,K}$ followed by the functor in Equation~\eqref{eq:Retraction} is naturally isomorphic to the identity, providing the desired retraction. 
By Lemma~\ref{lem:Embedding}, $\Phi_{\emptyset,K}$ is a representation embedding. 
The equivalence $\DD{\emptyset; \{k, \ell\} } \cong \fgMod{\kk[t_k,t_{\ell}]}$ from Example~\ref{ex:Extreme} together with Lemma~\ref{lem:PolyWild} concludes the proof.
\end{proof}

\begin{nota}
\begin{enumerate}
\item Let $\fgMod{\kk[t]}_{\leq n}$ denote the full subcategory of $\fgMod{\kk[t]}$ consisting of modules $M$ whose transition maps become isomorphisms past degree $n$:
\[
t^{e-d} \colon M(d) \ral{\cong} M(e) \quad \text{for } d \geq n. 
\]
\item For $1 \leq i \leq m$, denote by $\fgMod{R_{\si_i}}_{\leq n}$ the corresponding full subcategory of $\fgMod{R_{\si_i}}$ via the equivalence $\fgMod{R_{\si_i}} \cong \fgMod{\kk[t]}$ discussed in Remark~\ref{rem:GradedVect}. Explicitly: the $R_{\si_i}$-modules $M$ whose transition maps satisfy
\[
t^{\vec{e} - \vec{d}} \colon M(\vec{d}) \ral{\cong} M(\vec{e}) \quad \text{for } d_i \geq n. 
\]
\item Let $\DD{K_m}_{\leq n}$ denote the full subcategory of $\DD{K_m}$ such that for each $1 \leq i\leq m$, the $R_{\si_i}$-module $M_{\si_i}$ lies in $\fgMod{R_{\si_i}}_{\leq n}$.
\end{enumerate}
\end{nota}

\begin{rem}
The subcategory $\fgMod{\kk[t]}_{\leq n}$ consists of modules admitting a (finite) presentation with generators and relations in degrees $d \leq n$, which explains the notation.
\end{rem}

\begin{lem}\label{lem:ExactSubcat}
The subcategory $\DD{K_m}_{\leq n} \subset \DD{K_m}$ is closed under extensions.
\end{lem}

\begin{proof}
By Lemma~\ref{lem:PointwiseExact}, it suffices to show that the subcategory $\fgMod{R_{\si_i}}_{\leq n} \subset \fgMod{R_{\si_i}}$ is closed under extensions, which holds by the $5$-lemma.
\end{proof}

\begin{nota}
Given a quiver $Q$, let $\Rep_{\kk}(Q)$ denote the category of representations of the quiver $Q$ in $\kk$-vector spaces, and $\rep_{\kk}(Q)$ the full subcategory of pointwise finite-dimensional representations.

Let $Q_n$ denote the quiver which is a ``trivalent sink with legs of length $n$'', as illustrated here:

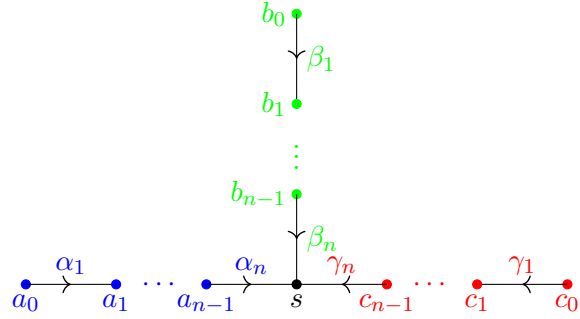
\begin{center}
\begin{tikzpicture}[scale=1]
\draw [black, fill=black] (0,0) circle (\VertexSize);
\node [below, black] at (0,0) {$s$};
\draw [blue, fill=blue] (-3,0) circle (\VertexSize);
\node [below, blue] at (-3,0) {$a_0$};
\draw [->-=0.5] (-3,0) -- (-2,0);
\node [above, blue] at (-2.5,0) {$\al_1$};
\draw [blue, fill=blue] (-2,0) circle (\VertexSize);
\node [below, blue] at (-2,0) {$a_1$};
\node [blue] at (-1.5,0) {$\cdots$};
\draw [blue, fill=blue] (-1,0) circle (\VertexSize);
\node [below, blue] at (-1,0) {$a_{n-1}$};
\draw [->-=0.5] (-1,0) -- (0,0);
\node [above, blue] at (-0.5,0) {$\al_n$};
\draw [green, fill=green] (0,3) circle (\VertexSize);
\node [left, green] at (0,3) {$b_0$};
\draw [->-=0.5] (0,3) -- (0,2);
\node [right, green] at (0,2.5) {$\be_1$};
\draw [green, fill=green] (0,2) circle (\VertexSize);
\node [left, green] at (0,2) {$b_1$};
\node [green] at (0,1.5) {$\vdots$};
\draw [green, fill=green] (0,1) circle (\VertexSize);
\node [left, green] at (0,1) {$b_{n-1}$};
\draw [->-=0.5] (0,1) -- (0,0);
\node [right, green] at (0,0.6) {$\be_n$};
\draw [red, fill=red] (3,0) circle (\VertexSize);
\node [below, red] at (3,0) {$c_0$};
\draw [->-=0.5] (3,0) -- (2,0);
\node [above, red] at (2.5,0) {$\ga_1$};
\draw [red, fill=red] (2,0) circle (\VertexSize);
\node [below, red] at (2,0) {$c_1$};
\node [red] at (1.5,0) {$\cdots$};
\draw [red, fill=red] (1,0) circle (\VertexSize);
\node [below, red] at (1,0) {$c_{n-1}$};
\draw [->-=0.5] (1,0) -- (0,0);
\node [above, red] at (0.6,0) {$\ga_n$};
\end{tikzpicture}
\end{center}
Note that $Q_n$ has $3n+1$ vertices. 
\end{nota}

Forgetting the orientation, %
the underlying graph of $Q_1$ is the Dynkin graph $D_4$; that of $Q_2$ is the Euclidean (or extended Dynkin) graph denoted $\tild{E}_6$ in the 
literature \cite{AssemSS06}*{\S VII.2} \cite{SimsonS07v2}*{\S XIII.1} \cite{DerksenW17}*{\S 4.2} \cite{ErdmannH18}*{\S 10.1}.

\begin{prop}\label{pr:QuiverRep}
There is an equivalence of categories $\DD{K_3}_{\leq n} \cong \rep_{\kk}(Q_n)$.
\end{prop}

\begin{proof}
We will obtain the equivalence as a composite of two equivalences
\[
\xymatrix{
\DD{K_3}_{\leq n} \ar[r]^-{\Psi}_-{\cong} & \rep_{\kk}(Q_{n+1})_{\leq n} \ar[r]^-{\Te}_-{\cong} & \rep_{\kk}(Q_n),
}
\]
where $\rep_{\kk}(Q_{n+1})_{\leq n}$ denotes the full subcategory of $\rep_{\kk}(Q_{n+1})$ where the maps past each vertex with index $n$ are isomorphisms, namely $\al_{n+1}$, $\be_{n+1}$, and $\ga_{n+1}$.

\textbf{The equivalence $\Psi$.} Given a module $M$ in $\DD{K_3}_{\leq n}$, consider the sequence of $\kk$-modules
\begin{equation}\label{eq:FirstLeg}
\xymatrix{
M_{\si_1}(0,n,n) \ar[r]^-{t_1} & M_{\si_1}(1,n,n) \ar[r]^-{t_1} & \cdots \ar[r]^-{t_1} & M_{\si_1}(n,n,n) \ar[r]^-{\phy_1}_{\cong} & M_{[3]}(n,n,n).
}
\end{equation}
The last map is an isomorphism, since the map of $R_{\si_1}$-modules $\phy_1 \colon M_{\si_1} \to M_{[3]}$ inverts $t_1$, and all the transition maps
\[
t_1^{d-n} \colon M_{\si_1}(n,n,n) \ral{\cong} M_{\si_1}(d,n,n)
\]
are isomorphisms for $d \geq n$ by assumption on $M$. Define the representation $\Psi(M)$ as having the sequence \eqref{eq:FirstLeg} as its first leg. Construct the second and third legs similarly using $M_{\si_2}$ and $M_{\si_3}$. This construction yields the desired functor $\Psi$, which is moreover an equivalence.

\textbf{The equivalence $\Te$.} Given a representation $V$ in $\rep_{\kk}(Q_{n+1})_{\leq n}$, construct the representation $\Te(V)$ by absorbing the last isomorphism into the previous map. That is, take the $\kk$-vector space $\Te(V)_s = V_s$ in the terminal position and
\[
\Te(V)_{\al_n} = V_{\al_{n+1}} \circ V_{\al_n} : V_{a_{n-1}} \to V_s,
\]
and likewise for the second and third legs. 
This yields a functor $\Te$ which is an equivalence, with inverse equivalence inserting an identity at the end of each leg.
\end{proof}

\begin{defn}
A finite quiver $Q$ has \Def{wild representation type} %
if its category of pointwise finite-dimensional representations $\rep_{\kk}(Q)$ does. Equivalently, the path algebra $\kk Q$ has wild representation type.
\end{defn}

Wildness of a quiver turns out to be independent of the base field $\kk$. That is, $Q$ has wild representation type over some field $\kk$ if and only if $Q$ has wild representation type over any other field $\kk'$ \cite{Nazarova73}. The same is true for a finite poset with a unique maximal element \cite{Nazarova75} \cite{ArnoldS05}*{\S 1}.

The argument in the next proof was kindly provided by Steffen Oppermann.

\begin{thm}\label{thm:WildRep}
For $m \geq 3$, the category $\DD{K_m}$ has wild representation type. 
\end{thm}

\begin{proof}
In the proof of Proposition~\ref{pr:Mgeq3}, we showed that the category $\DD{K_3}$ is an exact retract of $\DD{K_m}$. By %
Lemma~\ref{lem:Embedding}, 
it suffices to show that $\DD{K_3}$ has wild representation type. Its full subcategory $\DD{K_3}_{\leq n}$ %
is an exact subcategory, by Lemma~\ref{lem:ExactSubcat}. Thus it suffices to show that $\DD{K_3}_{\leq n}$ has wild representation type for $n$ large enough. We have an equivalence $\DD{K_3}_{\leq n} \cong \rep_{\kk}(Q_n)$ by Proposition~\ref{pr:QuiverRep}. For $n \geq 3$, the quiver $Q_n$ is known to have wild representation type \cite{SimsonS07v3}*{\S XVIII Theorem~4.1} \cite{Nazarova73}*{Lemma~8}. 
\end{proof}

In the remainder of the section, we interpret this result in terms of the torsion pair on $\DD{K_m}$ introduced in Section~\ref{sec:TorsionTh}.

\begin{rem}
The setup of Definition~\ref{def:Wild} can be generalized to categories $\cat{A}$ of the following form. Consider $\cat{A}'$ a $\kk$-linear abelian category and $\cat{A} \subseteq \cat{A}'$ a non-empty full %
subcategory that is closed under extensions (in particular finite direct sums) and summands. 
For example, if the abelian category $\cat{A}'$ has a torsion pair $(\cat{T}, \cat{F})$, then both classes $\cat{T}$ and $\cat{F}$ are subcategories of that form. 

The setup can be generalized further to a $\kk$-linear (Quillen) exact category $\cat{A}$ that is idempotent complete (i.e., in which every idempotent $e \colon X \to X$ splits). Background on exact categories can be found in \cite{Buhler10}. For background on idempotents and decompositions, see \cite{Shah23}*{\S 3}.
\end{rem}

\begin{lem}\label{lem:InducedTorsion}
Let $\cat{B}$ be an abelian category with a torsion pair $(\cat{T},\cat{F})$, and let $\cat{A} \subseteq \cat{B}$ be an abelian subcategory. Assume that for every object $A$ in $\cat{A}$, 
the monomorphism $i \colon T(A) \inj A$ in $\cat{B}$ 
lies in $\cat{A}$. %
Then the torsion pair on $\cat{B}$ restricts to a torsion pair $(\cat{T} \cap \cat{A}, \cat{F} \cap \cat{A})$ on $\cat{A}$.
\end{lem}

\begin{lem}
The torsion pair $(\cat{T}, \cat{F})$ on $\DD{K_m}$ induces a torsion pair $(\cat{T}_{\leq n}, \cat{F}_{\leq n})$ on the full subcategory $\DD{K_m}_{\leq n}$, for any $n \geq 0$.
\end{lem}

\begin{proof}
Let $M$ be an object in the subcategory $\DD{K_m}_{\leq n}$, i.e., $M_{\si_i}$ lies in $\fgMod{R_{\si_i}}_{\leq n}$ for all $1 \leq i \leq m$. Its torsion part $T(M)$ is given by
\[
T(M)_{\si_i} = T(M_{\si_i}),
\]
which also lies in $\fgMod{R_{\si_i}}_{\leq n}$, so that $T(M)$ lies in $\DD{K_m}_{\leq n}$. 
By Lemma~\ref{lem:InducedTorsion}, the claim follows.
\end{proof}

\begin{lem}
Via the equivalence from Proposition~\ref{pr:QuiverRep}, the torsion pair on $\DD{K_3}_{\leq n}$ corresponds to the two classes:
\begin{align*}
\cat{T}_{\leq n} &\cong \{ V \in \rep_{\kk}(Q_n) \mid V_s = 0 \} \\
\cat{F}_{\leq n} &\cong \{ V \in \rep_{\kk}(Q_n) \mid \text{the maps in $V$ are monomorphisms} \}.
\end{align*}
\end{lem}

\begin{cor}\label{cor:ProductAn}
The torsion class $\cat{T}_{\leq n}$ in $\DD{K_3}_{\leq n}$ is equivalent to a product of categories
\[
\cat{T}_{\leq n} \cong \rep_{\kk}(\vec{A}_n)^3,
\]
where $\vec{A}_n$ denotes the linearly ordered quiver of type $A_n$ (with $n$ vertices), as illustrated here:
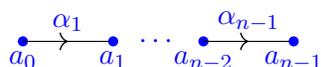
\begin{center}
\begin{tikzpicture}[scale=1.2]
\draw [blue, fill=blue] (-3,0) circle (\VertexSize);
\node [below, blue] at (-3,0) {$a_0$};
\draw [->-=0.5] (-3,0) -- (-2,0);
\node [above, blue] at (-2.5,0) {$\al_1$};
\draw [blue, fill=blue] (-2,0) circle (\VertexSize);
\node [below, blue] at (-2,0) {$a_1$};
\node [blue] at (-1.5,0) {$\cdots$};
\draw [blue, fill=blue] (-1,0) circle (\VertexSize);
\node [below, blue] at (-1,0) {$a_{n-2}$};
\draw [->-=0.5] (-1,0) -- (0,0);
\node [above, blue] at (-0.5,0) {$\al_{n-1}$};
\draw [blue, fill=blue] (0,0) circle (\VertexSize);
\node [below, blue] at (0,0) {$a_{n-1}$};
\end{tikzpicture}
\end{center}
\end{cor}

\begin{cor}\label{cor:TorsionFreeWild}
For $n \geq 3$, the torsion-free class $\cat{F}_{\leq n}$ in $\DD{K_3}_{\leq n}$ has wild representation type.
\end{cor}

\begin{proof}
Quivers of type $A_n$ have finite representation type, by Gabriel's theorem \cite{AssemSS06}*{\S VII Theorem~5.10} \cite{ErdmannH18}*{Theorem~11.1}. By Corollary~\ref{cor:ProductAn}, the torsion class $\cat{T}_{\leq n}$ has finite representation type. But the torsion pair $(\cat{T}_{\leq n}, \cat{F}_{\leq n})$ splits and $\DD{K_3}_{\leq n}$ has wild representation type, which forces %
$\cat{F}_{\leq n}$ to have wild representation type.
\end{proof}

\begin{rem}
A closer look at the proof that the quiver $Q_3$ has wild representation type shows Corollary~\ref{cor:TorsionFreeWild} directly. In \cite{Nazarova81}*{Theorem~2'} 
\cite{Nazarova75}, %
Nazarova shows that $Q_3$ is wild using representations in which all the $\kk$-linear maps are %
subspace inclusions.
\end{rem}

\section{Relationship to the rank invariant}\label{sec:Rank}

In this section, we work with the simplicial complex $K = K_m$ and investigate the relationship between the rank invariant of an $R$-module $M$ and the decomposition of the localization $L_{K_m}(M)$ in $\DD{K_m}$.

\begin{defn}
Let $M$ be an $R$-module. The \Def{rank invariant} of $M$ is the function assigning to each pair of multidegrees $\vec{a}, \vec{b} \in \N^m$ with $\vec{a} \leq \vec{b}$ the integer
\[
\rk_M (\vec{a},\vec{b}) = \rank \left( M(\vec{a}) \ral{t^{\vec{b} - \vec{a}}} M(\vec{b}) \right). 
\] 
\end{defn}

In general, the rank invariant does \emph{not} determine the decomposition of $L_{K_2}(M)$ in $\DD{K_2}$.

\begin{ex}\label{ex:SameRank}
Let $(t_1, t_2) \subset R$ be the %
ideal generated by $t_1$ and $t_2$ and take the $R$-modules
\begin{align*}
M &= (t_1, t_2) \op t_1 t_2 R \\
N &= t_1 R \op t_2 R. 
\end{align*}

\begin{figure}[H]
\begin{adjustbox}{width=0.6\textwidth}
\begin{tikzpicture}[scale=1]
 \begin{axis}[
axis lines=middle, 
ticks=none, 
axis equal, 
xmin=-0.2, xmax=5, ymin=-1, ymax=5]
\draw [blue, fill=blue] (0,1) circle (\VertexSize);
\draw [blue, fill=blue] (1,0) circle (\VertexSize);
\draw [blue, thick] (0,6) -- (0,1) -- (1,1) -- (1,0) -- (6,0);
\fill [blue, opacity=0.1] (0,6) -- (0,1) -- (1,1) -- (1,0) -- (6,0) -- (6,6) -- (0,6);
\draw [red, fill=red] (1,1) circle (\VertexSize);
\draw [red, thick] (1,6) -- (1,1) -- (6,1);
\fill [red, opacity=0.1] (1,1) rectangle (6,6);
\node at (3,-0.8) {$M$};
 \end{axis}
\end{tikzpicture}
\hspace{3ex}
\begin{tikzpicture}[scale=1]
 \begin{axis}[
axis lines=middle, 
ticks=none, 
axis equal, 
xmin=-0.2, xmax=5, ymin=-1, ymax=5]
\draw [blue, fill=blue] (0,1) circle (\VertexSize);
\draw [blue, thick] (0,6) -- (0,1) -- (6,1);
\fill [blue, opacity=0.1] (0,1) rectangle (6,6);
\draw [red, fill=red] (1,0) circle (\VertexSize);
\draw [red, thick] (1,6) -- (1,0) -- (6,0);
\fill [red, opacity=0.1] (1,0) rectangle (6,6);
\node at (2.7,-0.8) {$N$};
 \end{axis}
\end{tikzpicture}
\end{adjustbox}
\caption{The modules $M$ and $N$ in Example~\ref{ex:SameRank}.}
\label{fig:SameRank}
\end{figure}

The modules $M$ and $N$ have the same rank invariant:
\[
\rk_M (\vec{a},\vec{b}) = \begin{cases}
0 &\text{if } \vec{a} = (0,0) \\
1 &\text{if } \vec{a} = (a,0) \text{ or } \vec{a} = (0,a) \text{ for some } a > 0 \\
2 &\text{if } \vec{a} \geq (1,1).
\end{cases}
\]
However, their $K_2$-localizations are \emph{not} isomorphic in $\DD{K_2}$:
\begin{align*}
L_{K_2}(M) \cong [(0,0), \infty) \op [(1,1), \infty) \\
L_{K_2}(N) \cong [(1,0), \infty) \op [(0,1), \infty).
\end{align*}
\end{ex}

\begin{lem}\label{lem:FinPres}
Let $M$ be a finitely generated $R$-module. There exists $\vec{d} \in \N^{m}$ such that for every $\pvec{c}' \geq \vec{c} \geq \vec{d}$ in $\N^{m}$, the transition map of $M$
\[
M(\vec{c}) \ral{t^{\pvec{c}' - \vec{c}}} M(\pvec{c}')
\]
is an isomorphism.
\end{lem}

\begin{proof}
Since $M$ is finitely generated and $R$ is Noetherian, $M$ is finitely presented. Pick a finite presentation %
of $M$ and let $\vec{d} \in \N^m$ be an upper bound for the multidegrees of the generators and relations. Such a $\vec{d}$ satisfies the statement. 
\end{proof}

In the next statement, we will have a subset $\si \subseteq [m]$ and multidegrees viewed as functions %
$\vec{a} \in \N^{[m] \setminus \si}$ 
and %
$\vec{c} \in \N^{\si}$. 
Denote by %
$\vec{a} \ast \vec{c} \in \N^m$ 
the function extending $\vec{a}$ and $\vec{c}$.

\begin{lem}
Let $M$ be a finitely generated $R$-module and $\si \subseteq [m]$. For every %
$\vec{a}, \vec{b} \in \N^{[m] \setminus \si}$, 
the rank invariant of the localization $M_{\si} \dfn R_{\si} \ot_R M$ is given by
\begin{equation}\label{eq:RankLocaliz}
\rk_{M_{\si}}(\vec{a},\vec{b}) = \lim_{\vec{c} \in \N^{\si}} \rk_{M}(\vec{a} \ast \vec{c}, \vec{b} \ast \vec{c}).
\end{equation}
\end{lem}

\begin{proof}
For a fixed $\vec{a} \in \N^{[m] \setminus \si}$, apply Lemma~\ref{lem:FinPres} to the $\kk[t_i \mid i \in \si]$-module $M(\vec{a} \ast -)$. 
There exists a $\vec{d} = \vec{d}_{\vec{a}} \in \N^{\si}$ such that for every $\pvec{c}' \geq \vec{c} \geq \vec{d}$ in $\N^{\si}$, the transition map of $M$
\[
M(\vec{a} \ast \vec{c}) \ral{t^{\pvec{c}' - \vec{c}}} M(\vec{a} \ast \pvec{c}')
\]
is an isomorphism. The commutative square of transition maps
\[
\xymatrix{
M(\vec{a} \ast \vec{c}) \ar[d] \ar[r]^-{\cong} & M(\vec{a} \ast \pvec{c}') \ar[d] \\
M(\vec{b} \ast \vec{c}) \ar[r]^-{\cong} & M(\vec{b} \ast \pvec{c}') \\
}
\]
shows that the ranks on the right-hand side of Equation~\eqref{eq:RankLocaliz} are eventually constant, with the limit value being achieved for all $\vec{c} \geq \sup (\vec{d}_{\vec{a}}, \vec{d}_{\vec{b}})$. 
Moreover, the localization map $M \to M_{\si}$ induces an isomorphism 
\[
M(\vec{a} \ast \vec{c}) \ral{\cong} M_{\si}(\vec{a} \ast \vec{c})
\]
for all $\vec{c} \geq \vec{d}$, showing that the limit value in Equation~\eqref{eq:RankLocaliz} is indeed the rank $\rk_{M_{\si}}(\vec{a},\vec{b})$.
\end{proof}

\begin{cor}\label{cor:RankTorsion}
The rank invariant of an object $M$ of $\DD{K_m}$ determines the modules $M_{\si_i}$ and $M_{[m]}$ up to isomorphism.
\end{cor}

\begin{proof}
A finitely generated $\kk[t]$-module is determined by its rank invariant.
\end{proof}

\begin{ex}
For $m=2$, the rank invariant of an $R$-module $M$ determines the interval decompositions of the $\kk[t_1,t_2^{\pm}]$-module $M_{\si_1}$ and the $\kk[t_1^{\pm}, t_2]$-module $M_{\si_2}$. Each finite interval module $[a,b)_1$ in $M_{\si_1}$ contributes a ``vertical strip'' $[a,b)_1$ to the decomposition of $M$ in $\DD{K_2}$. Likewise, each finite interval module $[a,b)_2$ in $M_{\si_2}$ contributes a ``horizontal strip'' $[a,b)_2$ to the decomposition of $M$ in $\DD{K_2}$. The number of infinite interval modules $[a_i, \infty)_1$ in $M_{\si_1}$ equals the number of infinite interval modules in $M_{\si_2}$, which is $\dim_{\kk} M_{[m]}$. 
\end{ex}

\begin{prop}
For $m=2$, if an $R$-module $M$ lies in the image of the delocalization functor
$\rho_{K_2} \colon \DD{K_2} \to \fgMod{R}$ 
(see Lemma~\ref{lem:Delocalization}), then the rank invariant of $M$ determines the decomposition of $L_{K_2}(M)$ in $\DD{K_2}$.
\end{prop}

\begin{proof}
By Corollary~\ref{cor:RankTorsion}, the rank invariant of $M$ always determines the torsion part $T( L_{K_2}M )$ in $\DD{K_2}$. Since $M$ lies in the image of the right adjoint $\rho_{K_2}$, the unit map
$M \ral{\cong} \rho_{K_2} L_{K_2}(M)$ 
is an isomorphism. By additivity of the rank invariant under direct sum, we may assume that $L_{K_2}(M)$ is torsion-free. By Proposition~\ref{pr:M2Decompo}, $L_{K_2}(M)$ is a direct sum of ``quadrant modules'' $\bigoplus_i [\vec{d}^{(i)},\infty)$, so that $M$ is a free $R$-module
\[
M \cong \rho_{K_2} \left( \bigoplus_i [\vec{d}^{(i)},\infty) \right) \cong \bigoplus_i t^{\vec{d}^{(i)}} R.
\]
Such an $R$-module is determined (up to isomorphism) by its Hilbert function $\rk_M(\vec{a},\vec{a})$, in particular by its rank invariant.
\end{proof}

Our next goal is to find a refinement of the rank invariant of a $\kk[t_1,t_2]$-module $M$ that determines the torsion-free part of $L_{K_2}(M)$ in $\DD{K_2}$. 
An analogue of the %
following 
invariant was used in \cite{Zhang19} to decompose the representations of the quiver with relations consisting of a commutative square.

\begin{defn}\label{def:IntersectionRank}
Let $M$ be a $\kk[t_1,t_2]$-module. 
The \Def{intersection rank invariant} of $M$ assigns to any three bidegrees $(a_1,a_2), (b_1,b_2), (c_1,c_2) \in \N^2$ 
satisfying $(a_1,a_2) \leq (c_1,c_2)$ and $(b_1,b_2) \leq (c_1,c_2)$ the dimension of the intersection of images
\[
\irk_M((a_1,a_2), (b_1,b_2), (c_1,c_2)) = \dim_{\kk} \left( \im \left( M(\vec{a}) \ral{t^{\vec{c} - \vec{a}}} M(\vec{c}) \right) \cap \im \left( M(\vec{b}) \ral{t^{\vec{c} - \vec{b}}} M(\vec{c}) \right) \right). 
\]
The \Def{stabilized intersection rank invariant} of $M$ is the function $\sirk_M \colon \N^2 \to \N$ defined by 
\[
\sirk_M(a_1,b_2) = \lim_{a_2,b_1 \to \infty} \lim_{c_1,c_2 \to \infty} 
\irk((a_1,a_2), (b_1,b_2), (c_1,c_2)).
\]
\end{defn}

The formula is illustrated %
in Figure~\ref{fig:IntersectionRank}.

\begin{figure}[H]
\begin{tikzpicture}[scale=0.9]
 \begin{axis}[
axis lines=middle, 
ticks=none, 
axis equal, 
xmin=-0.5, xmax=7, ymin=-1, ymax=7.1]
\draw [blue, fill=blue] (1,4) circle (\VertexSize);
\draw [blue, fill=blue] (4,2) circle (\VertexSize);
\draw [red, fill=red] (6,6) circle (\VertexSize);
\draw [->,blue] (1,4) -- (6,6);
\draw [->,blue] (4,2) -- (6,6);
\draw [blue, dotted] (1,4) -- (1,0);
\node [below, blue] at (1,0) {$a_1$};
\draw [blue, dotted] (4,2) -- (0,2);
\node [left, blue] at (0,2) {$b_2$};
\node [above] at (1,4) {$M(\vec{a})$};
\node [below] at (4,2) {$M(\vec{b})$};
\node [right] at (6,6) {$M(\vec{c})$};
\node [above, blue] at (1,5) {$\vdots$};
\node [right, blue] at (5,2) {$\cdots$};
\node [red] at (7,7) {$\iddots$};
 \end{axis}
\end{tikzpicture}
\caption{The stabilized intersection rank of $M$.} %
\label{fig:IntersectionRank}
\end{figure}

\begin{lem}
\begin{enumerate}
\item Suppose $A \to M \to N \to B$ is an exact sequence of $R$-modules such that $A$ and $B$ are finite 
(i.e., finitely supported as diagrams $\N^2 \to \Vect{\kk}$). 
Then $\sirk_M = \sirk_N$. 
\item The function $\sirk_{(-)} \colon \Obj \left( \fgMod{R} \right) \to \{ f \colon \N^2 \to \N \}$ extends to 
\[
\sirk_{(-)} \colon \Obj \left( \DD{K_2} \right) \to \{ f\colon \N^2 \to \N \}.
\]
\end{enumerate}
\end{lem}

\begin{proof}
(1) Exactness of $R$-modules (viewed as diagrams $\N^2 \to \Vect{\kk}$) is pointwise, i.e., for every multidegree $\vec{d} \in \N^2$, the sequence of $\kk$-vector spaces
\[
A(\vec{d}) \to M(\vec{d}) \to N(\vec{d}) \to B(\vec{d})
\]
is exact. Since $A$ and $B$ are finite, the map of $R$-modules $M \to N$ is an isomorphism outside of a finite region in $\N^2$. In Definition~\ref{def:IntersectionRank}, letting $a_2,b_1 \in \N$ grow large enough shows that $\sirk_M$ is unchanged if we modify $M$ only in a finite region of $\N^2$.

(2) Two $R$-modules $M$ and $N$ become isomorphic in the Serre quotient $\DD{K}$ if and only if they are connected by a zigzag of $\ker(L_{K})$-equivalences, in fact the cospan
\[
\xymatrix{
M \ar[r]^-{\eta}_-{\sim} & \rho_{K} L_K (M) \ar[d]^{\cong} \\
N \ar[r]^-{\eta}_-{\sim} & \rho_{K} L_K (N) \\
}
\]
\cite{Popescu73}*{Proposition~4.4.3}. By part~(1) and Lemma~\ref{lem:SEquiv}, two $R$-modules connected by a $\ker(L_{K_2})$-equivalence have the same stabilized intersection rank invariant.
\end{proof}

\begin{lem}
\begin{enumerate}\label{lem:basic}
\item The stabilized intersection rank invariant of the $\kk[t_1,t_2]$-module $[\vec{d},\infty)$ is
\[
\sirk_{[\vec{d},\infty)}(a,b) = 
\begin{cases} 
1 & (a,b) \geq \vec{d} \\
0 & \text{otherwise.} \\ 
\end{cases}
\]
\item For $M,N \in \DD{K_2}$, we have $\sirk_{M \oplus N} = \sirk_M + \sirk_N$. 
\end{enumerate}
\end{lem}

We now show that the stabilized intersection rank invariant determines the torsion-free part of $L_{K_2}(M)$ in $\DD{K_2}$.

\begin{thm}\label{thm:RefinedRank}
If two torsion-free objects $M,N \in \DD{K_2}$ satisfy $\sirk_M = \sirk_N$, then they are isomorphic. In other words, the function 
\[
\sirk_{(-)} \colon \Obj \left( \cat{F} \DD{K_2} \right)/\mathrm{iso} \to \{f \colon \N^2 \to \N \}
\]
is injective. 
\end{thm}

\begin{proof}
Suppose the statement is false. Then there are torsion-free objects $M,N$ such that $\sirk_M = \sirk_N$ but $M \not\cong N$. 
By Proposition~\ref{pr:M2Decompo}, there are decompositions
\[
M \cong \bigoplus_{i=1}^m [\vec{d}^i, \infty) \quad \text{and} \quad N \cong \bigoplus_{j=1}^n [\vec{e}^j, \infty).
\]
We can assume that $M$ and $N$ 
have been chosen such that 
$m$ is the smallest possible. Let $\vec{d} \in \N^2$ be the smallest 
in lexicographic order 
such that $\sirk_M(\vec{d}) \neq 0$, then Lemma~\ref{lem:basic} implies that 
$M \cong [\vec{d}, \infty) \oplus M'$. Similarly $N \cong [\vec{d}, \infty) \oplus N'$. 
Again using Lemma~\ref{lem:basic}, $\sirk_{M'} = \sirk_{N'}$. Since~$M$ was the smallest counterexample, 
$M' \cong N'$ and so $M \cong N$, leading to a contradiction. 
\end{proof}

\begin{ex}
The $\kk[t_1,t_2]$-modules $M$ and $N$ in Example~\ref{ex:SameRank} satisfy
\[
\sirk_M(0,0) = 1 \quad \text{but} \quad \sirk_N(0,0) = 0,
\]
which distinguishes their $K_2$-localizations $L_{K_2}(M) \not\cong L_{K_2}(N)$.
\end{ex}

\section{Classification of tensor ideals}\label{sec:Classif}

In this section we work both in $\fgMod{R}$ and $\fgMod{R}_{\Z^m}$ and a module will mean a module in one of those categories. 
The methods of this section were developed in \cite{StanleyW11}, where more background material can be found. 

\begin{lem}
The homogeneous primes of $R$ are of the form $(t_{i_1}, \cdots , t_{i_k})$. 
\end{lem}

\begin{proof}
See \cite{HarringtonOST19}*{Lemma~4.35}.
\end{proof}

\begin{nota}
Let $\Spec(R)$ denote the set of homogeneous prime ideals of $R$. 
A subset $S$ of $\Spec(R)$ is closed if $p\in S$ and $p\subset q$ 
implies that $q\in S$. We denote the closure of $S$ by $\overline S$. 

The \Def{support} of a module $M$ is
$\Supp(M)=\{ \sigma \mid R_{\sigma^c}\otimes M \neq 0 \}$, where $\sigma^c$ is the complement of $\sigma\subset [m]$. We use $\Supp^c(M)=\{ \sigma \mid R_{\sigma}\otimes M \neq 0 \}$ for the set of complements of the support and also call it the support 
and consider it a subset of $\PP{[m]}$.

For each $\sigma\in \Supp^c(M)$ the corresponding prime ideal 
is $p = (t_i)_{i\not\in \sigma}$. In other words $p \in \Supp(M)$ if and only if $p^c \in \Supp^c(M)$.
We see that $\Supp^c(M)$ is closed under taking subsimplices and so a simplicial complex
since $\Supp(M)$ is a closed subset of $\Spec(R)$.  

We define $\Ass(M)$ to be the associated primes of $M$, 
by letting $p\in \Ass(M)$ if there exists $x\in M$ 
such that $p$ is the annihilator of $x$. 
\end{nota}

Here are some standard results about support and associated primes \cite{AtiyahM69}*{Exercise~3.19}. 

\begin{lem}\label{lem:suppass}
Let $M,N,L$ be modules and $0\to L\to M\to N\to 0$ be a short exact sequence. 
\begin{enumerate}
\item \label{item:ClosureAss} $\overline{\Ass(M)} = \Supp(M)$
\item $\Supp(L)\cup \Supp (N)=\Supp (M)$. 
\end{enumerate}
\end{lem}

\begin{nota}
For a subcategory $\cat{C}\subset \fgMod{R}$ (or $\fgMod{R}_{\Z^m}$), 
we define its \Def{support} as
$\Supp^c(\cat{C}) = \bigcup_{M\in \cat{C}} \Supp^c(M)$. 
Note that $\Supp^c(\cat{C})$ is also a simplicial complex. 
\end{nota}

\begin{defn}\label{def:TensorIdeal}
Let $\otimes$ be the tensor product on $\fgMod{R}_{\Z^m}$, which 
turns it into a symmetric monoidal category. 
\begin{enumerate}
 \item A Serre subcategory $\cat{C}\subset \fgMod{R}_{\Z^m}$ is 
a \Def{tensor ideal} if the conditions $M \in \cat{C}$ and $N \in \fgMod{R}_{\Z^m}$ 
imply $M \otimes N\in \cat{C}$. 
 \item A Serre subcategory $\cat{C} \subset \fgMod{R}$ 
is a \Def{tensor ideal} if the conditions $M \in \cat{C}$, $N\in \fgMod{R}_{\Z^m}$, 
and $M \otimes N \in \fgMod{R}$ imply $M \otimes N \in \cat{C}$. 
\end{enumerate}
Note that a tensor ideal in $\fgMod{R}$ is the 
same as $\cat{C} \cap \fgMod{R}$ for a tensor ideal $\cat{C}$ in $\fgMod{R}_{\Z^m}$. 
A tensor ideal is also called a \emph{tensor-closed Serre subcategory}.

For an object $M\in \fgMod{R}$, let $S(M)$ be the smallest Serre subcategory of 
$\fgMod{R}$ containing $M$ and $S^{\otimes}(M)$ be the smallest tensor ideal 
of $\fgMod{R}$ containing $M$. 
\end{defn}

\begin{lem}\label{lem:getprime}
For a prime $p$, $p\in \Supp(M)$ implies that for some 
$\vec{d}$, $t^{\vec{d}}R/p \in S(M)$ and thus $R/p\in S^{\otimes}(M)$
\end{lem}

\begin{proof}
By Lemma~\ref{lem:suppass}~\eqref{item:ClosureAss}, there is $q\subset p$ with $q\in \Ass(M)$. So for some 
$\vec{d}$ and $\alpha \in M(\vec{d})$ with anihilator $q$  we get that there is an injective map $t^{\vec{d}}R/q\to M$ and so since Serre subcategories are closed under taking subobjects $t^{\vec{d}}R/q\in S(M)$. Then since Serre subcategories are closed under quotients we get that  $t^{\vec{d}}R/p\in S(M)$ and so $R/p\in S^{\otimes}(M)$.
\end{proof}

\begin{lem}\label{lem:makeM}
$M\in S^{\otimes}(\bigoplus_{p\in \Supp(M)} R/p)$. 
\end{lem}

\begin{proof}
We construct a sequence of modules $M(i)$. Let $M(0)=M$. Suppose we have already constructed $M(i)$ and that $\Supp(M(i))\subset \Supp(M)$. If $M(i) \neq 0$ let $p\in \Ass(M(i))\subset \Supp (M(i))\subset \Supp(M)$ and define
$M(i+1)$ so there is an exact sequence
\begin{equation}\label{eq:Sses}
0\to t^{\vec{d}}R/p \to M(i)\to M(i+1)\to 0
\end{equation}
We also have $\Supp(M(i+1))\subset \Supp(M(i)) \subset \Supp(M)$ 
by Lemma~\ref{lem:suppass} and the induction hypothesis. 

The increasing sequence of submodules $\ker(M\to M(i))\subset M$ must stabilize and so 
$M(n)=0$ for some $n$. Then clearly $M(n)\in S^{\otimes}(\bigoplus_{p\in \Supp(M)} R/p)$.

As each $R/p$ used has $p\in \Supp M$ we get that $R/p\in S(\bigoplus_{p\in \Supp(M)} R/p)$
and $t^{\vec{d}}R/p \in S^{\otimes}(\bigoplus_{p\in \Supp(M)} R/p)$. 
Using Equation~\eqref{eq:Sses} and that Serre subcategories are closed under 
extensions we get that 
$M(n-1)\in S^{\otimes}(\bigoplus_{p\in \Supp(M)} R/p)$. Continuing in this way we see that
$M=M(0)\in S^{\otimes}(\bigoplus_{p\in \Supp(M)} R/p)$. 
\end{proof}

\begin{lem}\label{lem:closingM}
$\Supp(M)=\Supp(S(M))=\Supp(S^{\otimes}(M))$. 
\end{lem}
\begin{proof}
That $\Supp(M)\subset \Supp(S(M))\subset \Supp(S^{\otimes}(M))$ follow directly from the definitions. That $\Supp(S(M))\subset \Supp(M)$ follows since by Lemma~\ref{lem:suppass} the closure operations of a Serre subcategory, that is taking subobjects, quotients and extensions do not increase support. Also since $\Supp(M\otimes N)\subset \Supp(M)$ we get that 
$\Supp(S(M))\subset \Supp(S^{\otimes}(M))$.
\end{proof}

\begin{nota}
For $K$ a simplicial complex on $[m]$, recall the \Def{face ring} (or \Def{Stanley--Reisner ring}) of $K$, denoted $\SR{K}$. For a simplex $\{ i_1,\dots, i_k \} =\sigma\in \PP{[m]}$ we denote the monomial %
$t_{\sigma}=t_{i_1} \cdots t_{i_k}$. Then, let 
$I(K)= (t_{\sigma})_{\sigma\not\in K}$ and 
\[
\SR{K}=R/I(K).
\]
In other words, it is the (multigraded) polynomial ring modulo the 
ideal $I(K)$ generated by missing faces. Remembering we work with finitely generated modules, we call a module $M$, $I$-torsion if $I^s M=0$ for some $s$. We let 
$\tors{I}$ denote the full subcategory of $I$-torsion modules.  
\end{nota}

\begin{lem}\label{lem:same}
$\ker L_K = \tors{I(K)} = S^{\otimes}(\SR{K}) = S^{\otimes}(\oplus_{p^c\in K} R/p)$.
\end{lem}

\begin{proof}
First observe that $M_{\sigma}=0 \iff (t_{\sigma})^s M = 0$ for some $s$. 
Then 
\begin{align*}
M\in \ker L_K  & \iff  M_{\sigma}=0 \mbox{ for all } \sigma\not\in K
\\
& \iff t_{\sigma}^s M = 0 \mbox{ for all } \sigma\not\in K 
\mbox{ for some } s 
\\
& \iff (I(K))^s M = 0 \mbox{ for some } s
\\
& \iff M \mbox{ is in } \tors{I(K)}.
\end{align*}
This proves the first equality. The last equality follows from Lemma~\ref{lem:RiseStan} and the computation 
$\Supp^c(\oplus_{p^c\in K} R/p) = K$ since by Theorem~\ref{thm:Classif} $\Supp^c$ determines tensor ideals in $\fgMod{R}$ ($\fgMod{R}_{\Z^m}$). Since for each $p^c \in K$ and $\sigma \not\in K$, 
$R_{\sigma}\otimes R/p=0$ and so $R/p^c\in \ker L_K$. It follows that 
$S^{\otimes}(\oplus_{p^c\in K} R/p)\subset \ker L_K$. 
If $M\in \ker L_K$ then $M_\sigma=0$ for each $\sigma \not\in K$ so $\Supp^c M\subset K$, 
thus by Lemma~\ref{lem:makeM}, %
$M \in S^{\otimes}(\oplus_{p^c\in K} R/p)$. This shows that the first class is equal to the last. %
\end{proof}

\begin{rem}
The open subspace complement of the closed subset $V(K)$ 
cut out by the ideal $I(K)$ is written in polyhedral product language as $(\C, \C^*)^K$. The homotopy quotient 
of  $(\C, \C^*)^K$ by $(S^1)^m$ is the Davis--Januszkiewicz space $DJ(K)$. 
The cohomology algebra $H^*(DJ(K);\kk)$, and therefore the equivariant cohomology of $(\C, \C^*)^K$, are both isomorphic to $\SR{K}$, the Stanley--Reisner ring of $K$ \cite{BuchstaberP15}*{\S 4.3, 4.7}. Coincidentally the category of $(\C^*)^m$-equivariant coherent sheaves on  $(\C, \C^*)^K$ is 
related to $\fgMod{R}_{\Z^m}$ modulo the tensor ideal %
generated by $H^*(DJ(K);\kk)$. 
\end{rem}

\begin{lem}\label{lem:RiseStan}
$\Supp^c(\SR{K}) = K$.
\end{lem}

\begin{proof}
We have an exact sequence 
\[
\bigoplus_{\sigma\not\in K} t_{\sigma} R
\stackrel{\Sigma_{\sigma\not\in K} f_{\sigma}}{\rightarrow}
R \rightarrow \SR{K}\rightarrow 0
\]
where $f_{\sigma}\colon  t_{\sigma} R \rightarrow R$ 
is the inclusion. 
Tensoring with $R_{\delta}$ preserves this exact sequence, so 
$\SR{K}\otimes R_{\delta}=0$ if and only if 
$\Sigma_{\sigma\not\in K} f_{\sigma}\otimes R_{\delta}$ is a surjection. 

Let $\sigma\subseteq\delta \subseteq [m]$. Recall that $R_{\delta}$ then inverts all the 
elements in $\delta$. So all $t_i$ that make up $t_{\sigma}$ are inverted and thus
$f_{\sigma}\otimes R_{\delta}$ becomes an isomorphism. %
If there is a missing face $\sigma \not\in K$ satisfying $\sigma \subseteq \delta$, then the map 
$\Sigma_{\sigma\not\in K} f_{\sigma}\otimes R_{\delta}$ is a surjection and 
$\SR{K}\otimes R_{\delta}=0$. So if the inclusion $\sigma \subseteq \delta$ holds for some missing face 
$\sigma \not\in K$, then $\delta \not\in \Supp^c(\SR{K})$. Also if $\delta \not\in K$, then there is a $\sigma \not\in K$ with $\sigma \subseteq \delta$. Thus $\Supp^c(\SR{K}) \subseteq K$.

On the other hand if $\sigma \not\subseteq \delta$, then 
not all the $t_i$ that make up $t_{\sigma}$ are inverted and so 
$t_{\sigma}R_{\sigma}\otimes R_{\delta}(\vec{0}) = 0$. 
This implies that if for every missing face $\sigma \not\in K$, $\sigma \not\subseteq \delta$ then 
$\Sigma_{\sigma\not\in K} f_{\sigma}\otimes  R_{\delta}$ is not surjective,
$\SR{K}\otimes R_{\delta} \neq 0$ and so $\delta \in \Supp^c(\SR{K})$. Observe that any $\delta \in K$ will have this property. Therefore
$K \subseteq \Supp^c(\SR{K})$, and the proof is complete. 
\end{proof}

\begin{lem}\label{lem:supptest}
For any tensor ideal $\cat{C}\subseteq\fgMod{R}$ and $R$-module, M
\begin{equation*}
M \in \cat{C} \iff \Supp(M) \subseteq \Supp(\cat{C}).
\end{equation*}
\end{lem}

\begin{proof}
If $M \in \cat{C}$ then clearly $\Supp(M)\subseteq \Supp(\cat{C})$. 

Next assume that $\Supp(M) \subseteq \Supp(\cat{C})$. Then for each $p\in \Supp(M)$ there is 
$N \in \cat{C}$ such that $p \in \Supp(N)$. Then by Lemma~\ref{lem:getprime} $R/p \in \cat{C}$. 
Since $\cat{C}$ is closed under extensions it is closed under 
finite direct sums. Then since $\Supp(M)$ is finite we get that 
$\bigoplus_{p\in \Supp(M)} R/p\in \cat{C}$. So from Lemma~\ref{lem:makeM} $M\in \cat{C}$. 
\end{proof}

\begin{thm}\label{thm:Classif}
$\Supp^c$ induces a bijection (of sets)
\[
\mbox{tensor ideals in } \fgMod{R}\ (
\mbox{or } \fgMod{R}_{\Z^m}) \to
\mbox{simplicial complexes on }[m]
\]
with inverse $K \mapsto S^{\otimes}(\SR{K})$. 
\end{thm}

\begin{proof}
For a simplicial complex $K$ on $[m]$, we have 
\begin{equation*}
\Supp^c(S^{\otimes}(\SR{K})) = \Supp^c(\SR{K}) = K,
\end{equation*}
the first equality being Lemma~\ref{lem:closingM}  
and the second being Lemma~\ref{lem:RiseStan}. 
We conclude that $\Supp\circ S$ is the identity. 
The equation also implies that for any %
tensor ideal $\cat{C}$, 
\begin{equation*} 
\Supp \cat{C} = \Supp S^{\otimes}(\SR{\Supp \cat{C}}).
\end{equation*}
Lemma~\ref{lem:supptest} then shows that $M\in \cat{C}$ if 
and only if $M\in S (\Supp \cat{C})$, which shows $S\circ \Supp$ is 
also the identity. %
\end{proof}

\section{Simples}\label{sec:Simple}

For this section $K$ is again any simplicial complex on $[m]$. 
In this section we will classify the simples in $\DD{K}$ and 
use that classification to show that when $K$ is 
a skeleton of $\De^{m-1}$, $\DD{K}$ %
is obtained from $\fgMod{R}$ by 
iteratively quotienting out the simples. 

A subset $\sigma\subset [m]$ is a \Def{minimal missing face} of $K$ if 
$\sigma\not\in K$ and for any $\tau\subsetneq \sigma$, $\tau\in K$. 
For $K=\emptyset$,  $\emptyset$ is considered a minimal missing face. 

For a minimal missing face $\sigma$ of $K$, consider the $R_{\sigma}$-module, 
$S(\sigma)=R_{\sigma}/(t_i)_{i\not\in \sigma}$. 
Let $S(\sigma)\in \DD{K}$ be given by the formula
\begin{equation}\label{eq:simpleform}
S(\sigma)_{\tau} = 
\begin{cases}
R_{\sigma}/(t_i)_{i\notin \sigma} & \text{if } \tau = \sigma \\
0 & \text{otherwise.}
\end{cases}
\end{equation}
The following alternative way of looking at $S(\sigma)$ also 
shows the formulas define an element of $\DD{K}$. 

\begin{lem}
For a minimal missing face $\sigma\in K$, 
$S(\sigma)\cong L_K R/(t_i)_{i\not\in \sigma}$.
\end{lem}

\begin{lem}
For any minimal missing face $\sigma$ of $K$ and $\vec{d}\in \N^m$,
$t^{\vec{d}} S(\sigma)\in \DD{K}$ is simple. 
\end{lem}

\begin{proof}
Observe that $S(\sigma)_{\sigma}=R_{\sigma}/(t_i)_{i\not\in \sigma}$ is a simple 
$R_{\sigma}$-module. Using Lemma~\ref{lem:PointwiseExact} this implies that $S(\sigma)$ is simple, 
and similarly $t^{\vec{d}} S(\sigma)\in \DD{K}$ is simple. 
\end{proof}

Next we show that these are all the simples in $\DD{K}$. 

\begin{lem}\label{lem:simp}
Any simple in $\DD{K}$ is isomorphic to $t^{\vec{d}} S(\sigma)$ 
for some minimal missing face $\sigma$ and some $\vec{d}\in \N^m$.
\end{lem}

\begin{proof}
Suppose $M\in \DD{K}$ is simple. For some $\sigma$ we 
have that $M_{\sigma} \neq 0$. By taking points away from  $\sigma$ 
if necessary we can get a minimal 
missing face $\tau$ with $\tau\subset \sigma$.  Since 
$M_{\sigma}\cong M_{\tau}\otimes_{R_{\tau}} R_{\sigma}$, we 
get that $M_{\tau} \neq 0$.

Next let $x\in M_{\tau}$ be a non-zero element. 
For $i\not\in \tau$, $t_i M_{\tau}$ is a proper submodule of 
$M_{\tau}$ and so 
$t_i M$ is a proper submodule of $M$. Since $M$ is simple 
this proper submodule must be the $0$ module, and thus 
$t_i M_{\tau}=0$ and $t_i x=0$. 

Suppose $x\in M_{\tau}(\vec{d})$ (recall that means $x$ is in multidegree 
$\vec{d}$). Then consider the map
$t^{\vec{d}} R_{\sigma}/(t_i)_{i\not\in \sigma}  \rightarrow M_{\tau}$
 that sends $1$ to $x$. 
Since $x \neq 0$, this extends to a non-zero map 
$t^{\vec{d}}S(\tau)\rightarrow M$, 
which must be an isomorphism since both 
$t^{\vec{d}}S(\tau)$ and $M$ are simple. 
\end{proof}

The following definition coincides with that in \cite{Gabriel62}*{\S IV.1} for categories of finitely generated modules and the categories $\DD{K}$.

\begin{defn}
Given an abelian category $\mathcal{A}$, we let $\mathcal{A}=\mathcal{A}(-1)$ and 
$\mathcal{A}(i+1)=\mathcal{A}(i)/ S(\mathcal{A}(i))$ where $S(\mathcal{A})$ is the 
Serre subcategory generated by the simples in $\mathcal{A}$. 
The \Def{Krull dimension} of $\mathcal{A}$, $\Kdim(A)$ is the least $n$ 
such that $\mathcal{A}(n)$ is equivalent to the $0$-category 
(i.e.\ the abelian category with one object). 
\end{defn}

\begin{ex}
The trivial abelian category $0$ has Krull dimension $-1$. We have $\Kdim(\vect{\kk}) = 0$ and $\Kdim(\fgMod{\kk[t_1, \ldots, t_m]}) = m$.
\end{ex}

\begin{prop}
The category $\DD{K}$ has Krull dimension $\Kdim(\DD{K}) = m-s$ where $s$ is the smallest %
cardinality 
of a missing face of $K$. 
\end{prop}

In other words, $\Kdim(\DD{K})$ is the largest number of non-inverted variables $t_i$ appearing in the rings $R_{\si}$ for missing faces $\si \in K^c$.

\begin{proof}
Using Lemma~\ref{lem:simp}, the simples in $S(\DD{K})$ are $t^{\vec{d}} S(\sigma)$ where $\sigma$ runs over the minimal missing faces of $K$. Let $K'$ denote $K$ together with its minimal missing faces.
By Lemma~\ref{lem:same} $\DD{K}\cong \fgMod{R}/S^{\otimes}(\oplus_{p^c\in K} R/p)$, which yields
\[
\DD{K}/S(\DD{K})\cong \fgMod{R}/S^{\otimes}(\oplus_{p^c\in K'} R/p)=\DD{K'}.
\]
This implies that for $n=m-s$, %
$\DD{K}(n)\cong \DD{\Delta^{m-1}}\cong 0$ and so $\Kdim(\DD{K})\leq n$. 
Also we have 
\[
\DD{K}(n-1)\cong \DD{\partial \Delta^{m-1}} \cong \vect{\kk} \not\cong 0,
\]
which shows $\Kdim(\DD{K})\geq n$, hence $\Kdim(\DD{K})=n$.
\end{proof}

\begin{cor}\label{cor:ModOutSimples}
For the simplicial complex $K = \sk_{i} \De^{m-1}$ with $-2 \leq i \leq m-1$, the category $\DD{K}$ is obtained from $\fgMod{R}$ by iteratively quotienting out the simples %
$i+2$ times. In particular for $K_m = \sk_{m-3} \De^{m-1}$, $\DD{K_m}$ is obtained from $\fgMod{R}$ by %
quotienting out the simples %
$m-1$ times.
\end{cor}

\begin{rem}
In the setup of \cite{BubenikSS22}, if we give the simples weight $1$, the Serre subcategory generated by the simples consists of the objects with finite distance from the $0$ object.
\end{rem}

\vspace*{8pt}


\begin{bibdiv}
%
\begin{biblist}

\bib{ArnoldS05}{article}{
  author={Arnold, David M.},
  author={Simson, Daniel},
  title={Endo-wild representation type and generic representations of finite posets},
  journal={Pacific J. Math.},
  volume={219},
  date={2005},
  number={1},
  pages={1--26},
  issn={0030-8730},
  review={\MR {2174218}},
  doi={10.2140/pjm.2005.219.1},
}

\bib{AsashibaBENY22}{article}{
  author={Asashiba, Hideto},
  author={Buchet, Micka\"{e}l},
  author={Escolar, Emerson G.},
  author={Nakashima, Ken},
  author={Yoshiwaki, Michio},
  title={On interval decomposability of 2D persistence modules},
  journal={Comput. Geom.},
  volume={105},
  date={2022},
  pages={Paper No. 101879, 33},
  issn={0925-7721},
  review={\MR {4402576}},
  doi={10.1016/j.comgeo.2022.101879},
}

\bib{AssemSS06}{book}{
  author={Assem, Ibrahim},
  author={Simson, Daniel},
  author={Skowro\'{n}ski, Andrzej},
  title={Elements of the representation theory of associative algebras. Vol. 1},
  series={London Mathematical Society Student Texts},
  volume={65},
  note={Techniques of representation theory},
  publisher={Cambridge University Press, Cambridge},
  date={2006},
  pages={x+458},
  isbn={978-0-521-58423-4},
  isbn={978-0-521-58631-3},
  isbn={0-521-58631-3},
  review={\MR {2197389}},
  doi={10.1017/CBO9780511614309},
}

\bib{AtiyahM69}{book}{
  author={Atiyah, M. F.},
  author={Macdonald, I. G.},
  title={Introduction to commutative algebra},
  publisher={Addison-Wesley Publishing Co., Reading, Mass.-London-Don Mills, Ont.},
  date={1969},
  pages={ix+128},
  review={\MR {0242802}},
}

\bib{BauerBOS20}{article}{
  author={Bauer, Ulrich},
  author={Botnan, Magnus B.},
  author={Oppermann, Steffen},
  author={Steen, Johan},
  title={Cotorsion torsion triples and the representation theory of filtered hierarchical clustering},
  journal={Adv. Math.},
  volume={369},
  date={2020},
  pages={107171, 51},
  issn={0001-8708},
  review={\MR {4091895}},
  doi={10.1016/j.aim.2020.107171},
}

\bib{BlanchetteBH22}{article}{
  author={Blanchette, Benjamin},
  author={Br\"{u}stle, Thomas},
  author={Hanson, Eric J.},
  title={Homological approximations in persistence theory},
  journal={Canad. J. Math.},
  date={2022},
  pages={1--38},
  doi={10.4153/S0008414X22000657},
}

\bib{Borceux94v2}{book}{
  author={Borceux, Francis},
  title={Handbook of Categorical Algebra 2: Categories and Structures},
  series={Encyclopedia of Mathematics and its Applications},
  volume={51},
  publisher={Cambridge University Press},
  date={1994},
}

\bib{BotnanC20}{article}{
  author={Botnan, Magnus Bakke},
  author={Crawley-Boevey, William},
  title={Decomposition of persistence modules},
  journal={Proc. Amer. Math. Soc.},
  volume={148},
  date={2020},
  number={11},
  pages={4581--4596},
  issn={0002-9939},
  review={\MR {4143378}},
  doi={10.1090/proc/14790},
}

\bib{BotnanLO22}{article}{
  author={Botnan, Magnus Bakke},
  author={Lebovici, Vadim},
  author={Oudot, Steve},
  title={On rectangle-decomposable 2-parameter persistence modules},
  journal={Discrete Comput. Geom.},
  volume={68},
  date={2022},
  number={4},
  pages={1078--1101},
  issn={0179-5376},
  review={\MR {4517095}},
  doi={10.1007/s00454-022-00383-y},
}

\bib{BotnanLO23}{article}{
  author={Botnan, Magnus Bakke},
  author={Lebovici, Vadim},
  author={Oudot, Steve},
  title={Local characterizations for decomposability of 2-parameter persistence modules},
  journal={Algebr. Represent. Theory},
  date={2023},
  doi={10.1007/s10468-022-10189-4},
}

\bib{BotnanL23}{article}{
  author={Botnan, Magnus Bakke},
  author={Lesnick, Michael},
  title={An introduction to multiparameter persistence},
  conference={ title={Representations of algebras and related structures}, },
  book={ series={EMS Ser. Congr. Rep.}, publisher={EMS Press, Berlin}, },
  isbn={978-3-98547-054-9},
  date={2023},
  pages={77--150},
  review={\MR {4693638}},
}

\bib{BotnanOO22}{article}{
  author={Botnan, Magnus Bakke},
  author={Oppermann, Steffen},
  author={Oudot, Steve},
  title={Signed barcodes for multi-parameter persistence via rank decompositions},
  conference={ title={38th International Symposium on Computational Geometry}, },
  book={ series={LIPIcs. Leibniz Int. Proc. Inform.}, volume={224}, publisher={Schloss Dagstuhl. Leibniz-Zent. Inform., Wadern}, },
  date={2022},
  pages={Paper No. 19, 18},
  review={\MR {4470898}},
  doi={10.4230/lipics.socg.2022.19},
}

\bib{BubenikSS22}{article}{
  author={Bubenik, Peter},
  author={Scott, Jonathan},
  author={Stanley, Donald},
  title={Exact weights, path metrics, and algebraic Wasserstein distances},
  journal={J Appl. and Comput. Topology},
  doi={10.1007/s41468-022-00103-8},
  date={2022},
}

\bib{BuchstaberP15}{book}{
  author={Buchstaber, Victor M.},
  author={Panov, Taras E.},
  title={Toric topology},
  series={Mathematical Surveys and Monographs},
  volume={204},
  publisher={American Mathematical Society, Providence, RI},
  date={2015},
  pages={xiv+518},
  isbn={978-1-4704-2214-1},
  review={\MR {3363157}},
  doi={10.1090/surv/204},
}

\bib{Buhler10}{article}{
  author={B\"{u}hler, Theo},
  title={Exact categories},
  journal={Expo. Math.},
  volume={28},
  date={2010},
  number={1},
  pages={1--69},
  issn={0723-0869},
  review={\MR {2606234}},
  doi={10.1016/j.exmath.2009.04.004},
}

\bib{CarlssonZ05}{article}{
  author={Carlsson, Gunnar},
  author={Zomorodian, Afra},
  title={Computing persistent homology},
  journal={Discrete Comput. Geom.},
  volume={33},
  date={2005},
  number={2},
  pages={249--274},
  issn={0179-5376},
  review={\MR {2121296}},
  doi={10.1007/s00454-004-1146-y},
}

\bib{CarlssonZ09}{article}{
  author={Carlsson, Gunnar},
  author={Zomorodian, Afra},
  title={The theory of multidimensional persistence},
  journal={Discrete Comput. Geom.},
  volume={42},
  date={2009},
  number={1},
  pages={71--93},
  issn={0179-5376},
  review={\MR {2506738}},
  doi={10.1007/s00454-009-9176-0},
}

\bib{ChacholskiCS24}{article}{
  author={Chach\'olski, Wojciech},
  author={Corbet, Ren\'e},
  author={Sattelberger, Anna-Laura},
  title={The shift-dimension of multipersistence modules},
  journal={J. Appl. Comput. Topol.},
  volume={8},
  date={2024},
  number={3},
  pages={643--667},
  issn={2367-1726},
  review={\MR {4799025}},
  doi={10.1007/s41468-024-00169-6},
}

\bib{ClauseDMW23}{article}{
  author={Clause, Nate},
  author={Dey, Tamal K.},
  author={M\'emoli, Facundo},
  author={Wang, Bei},
  title={Meta-diagrams for 2-parameter persistence},
  conference={ title={39th International Symposium on Computational Geometry}, },
  book={ series={LIPIcs. Leibniz Int. Proc. Inform.}, volume={258}, publisher={Schloss Dagstuhl. Leibniz-Zent. Inform., Wadern}, },
  isbn={978-3-95977-273-0},
  date={2023},
  pages={Art. No. 25, 16},
  review={\MR {4604010}},
  doi={10.4230/lipics.socg.2023.25},
}

\bib{Cox95}{article}{
  author={Cox, David A.},
  title={The homogeneous coordinate ring of a toric variety},
  journal={J. Algebraic Geom.},
  volume={4},
  date={1995},
  number={1},
  pages={17--50},
  issn={1056-3911},
  review={\MR {1299003}},
}

\bib{CoxLS11}{book}{
  author={Cox, David A.},
  author={Little, John B.},
  author={Schenck, Henry K.},
  title={Toric varieties},
  series={Graduate Studies in Mathematics},
  volume={124},
  publisher={American Mathematical Society, Providence, RI},
  date={2011},
  pages={xxiv+841},
  isbn={978-0-8218-4819-7},
  review={\MR {2810322}},
  doi={10.1090/gsm/124},
}

\bib{DerksenW17}{book}{
  author={Derksen, Harm},
  author={Weyman, Jerzy},
  title={An introduction to quiver representations},
  series={Graduate Studies in Mathematics},
  volume={184},
  publisher={American Mathematical Society, Providence, RI},
  date={2017},
  pages={x+334},
  isbn={978-1-4704-2556-2},
  review={\MR {3727119}},
  doi={10.1090/gsm/184},
}

\bib{Dickson66}{article}{
  author={Dickson, Spencer E.},
  title={A torsion theory for Abelian categories},
  journal={Trans. Amer. Math. Soc.},
  volume={121},
  date={1966},
  pages={223--235},
  issn={0002-9947},
  review={\MR {191935}},
  doi={10.2307/1994341},
}

\bib{ErdmannH18}{book}{
  author={Erdmann, Karin},
  author={Holm, Thorsten},
  title={Algebras and representation theory},
  series={Springer Undergraduate Mathematics Series},
  publisher={Springer, Cham},
  date={2018},
  pages={ix+298},
  isbn={978-3-319-91997-3},
  isbn={978-3-319-91998-0},
  review={\MR {3837168}},
  doi={10.1007/978-3-319-91998-0},
}

\bib{Gabriel62}{article}{
  author={Gabriel, Pierre},
  title={Des cat\'{e}gories ab\'{e}liennes},
  language={French},
  journal={Bull. Soc. Math. France},
  volume={90},
  date={1962},
  pages={323--448},
  issn={0037-9484},
  review={\MR {232821}},
}

\bib{GafvertC21}{article}{
  author={G\"afvert, Oliver},
  author={Chach\'{o}lski, Wojciech},
  title={Stable Invariants for Multiparameter Persistence},
  date={2021},
  eprint={arXiv:1703.03632},
  status={Preprint},
}

\bib{GelfandM03}{book}{
  author={Gelfand, Sergei I.},
  author={Manin, Yuri I.},
  title={Methods of homological algebra},
  series={Springer Monographs in Mathematics},
  edition={2},
  publisher={Springer-Verlag, Berlin},
  date={2003},
  pages={xx+372},
  isbn={3-540-43583-2},
  review={\MR {1950475 (2003m:18001)}},
  doi={10.1007/978-3-662-12492-5},
}

\bib{HarringtonOST19}{article}{
  author={Harrington, Heather A.},
  author={Otter, Nina},
  author={Schenck, Hal},
  author={Tillmann, Ulrike},
  title={Stratifying multiparameter persistent homology},
  journal={SIAM J. Appl. Algebra Geom.},
  volume={3},
  date={2019},
  number={3},
  pages={439--471},
  review={\MR {4000203}},
  doi={10.1137/18M1224350},
}

\bib{Hartshorne77}{book}{
  author={Hartshorne, Robin},
  title={Algebraic geometry},
  note={Graduate Texts in Mathematics, No. 52},
  publisher={Springer-Verlag, New York-Heidelberg},
  date={1977},
  pages={xvi+496},
  isbn={0-387-90244-9},
  review={\MR {0463157}},
}

\bib{KashiwaraS06}{book}{
  author={Kashiwara, Masaki},
  author={Schapira, Pierre},
  title={Categories and sheaves},
  series={Grundlehren der Mathematischen Wissenschaften [Fundamental Principles of Mathematical Sciences]},
  volume={332},
  publisher={Springer-Verlag, Berlin},
  date={2006},
  pages={x+497},
  isbn={978-3-540-27949-5},
  isbn={3-540-27949-0},
  review={\MR {2182076 (2006k:18001)}},
  doi={10.1007/3-540-27950-4},
}

\bib{KimM24}{article}{
  author={Kim, Woojin},
  author={Moore, Samantha},
  title={Bigraded Betti numbers and generalized persistence diagrams},
  journal={J. Appl. Comput. Topol.},
  volume={8},
  date={2024},
  number={3},
  pages={727--760},
  issn={2367-1726},
  review={\MR {4799028}},
  doi={10.1007/s41468-024-00180-x},
}

\bib{Lesnick15}{article}{
  author={Lesnick, Michael},
  title={The theory of the interleaving distance on multidimensional persistence modules},
  journal={Found. Comput. Math.},
  volume={15},
  date={2015},
  number={3},
  pages={613--650},
  issn={1615-3375},
  review={\MR {3348168}},
  doi={10.1007/s10208-015-9255-y},
}

\bib{Nazarova73}{article}{
  author={Nazarova, L. A.},
  title={Representations of quivers of infinite type},
  language={Russian},
  journal={Izv. Akad. Nauk SSSR Ser. Mat.},
  volume={37},
  date={1973},
  pages={752--791},
  issn={0373-2436},
  review={\MR {0338018}},
}

\bib{Nazarova75}{article}{
  author={Nazarova, L. A.},
  title={Partially ordered sets of infinite type},
  language={Russian},
  journal={Izv. Akad. Nauk SSSR Ser. Mat.},
  volume={39},
  date={1975},
  number={5},
  pages={963--991, 1219},
  issn={0373-2436},
  review={\MR {0406878}},
}

\bib{Nazarova81}{article}{
  author={Nazarova, L. A.},
  title={Poset representations},
  conference={ title={Integral representations and applications}, address={Oberwolfach}, date={1980}, },
  book={ series={Lecture Notes in Math.}, volume={882}, publisher={Springer, Berlin-New York}, },
  date={1981},
  pages={345--356},
  review={\MR {646110}},
}

\bib{Oudot15}{book}{
  author={Oudot, Steve Y.},
  title={Persistence theory: from quiver representations to data analysis},
  series={Mathematical Surveys and Monographs},
  volume={209},
  publisher={American Mathematical Society, Providence, RI},
  date={2015},
  pages={viii+218},
  isbn={978-1-4704-2545-6},
  review={\MR {3408277}},
  doi={10.1090/surv/209},
}

\bib{OudotS24}{article}{
  author={Oudot, Steve},
  author={Scoccola, Luis},
  title={On the stability of multigraded Betti numbers and Hilbert functions},
  journal={SIAM J. Appl. Algebra Geom.},
  volume={8},
  date={2024},
  number={1},
  pages={54--88},
  review={\MR {4695718}},
  doi={10.1137/22M1489150},
}

\bib{Perling04}{article}{
  author={Perling, Markus},
  title={Graded rings and equivariant sheaves on toric varieties},
  journal={Math. Nachr.},
  volume={263/264},
  date={2004},
  pages={181--197},
  issn={0025-584X},
  review={\MR {2029751}},
  doi={10.1002/mana.200310130},
}

\bib{Perling13}{article}{
  author={Perling, Markus},
  title={Resolutions and cohomologies of toric sheaves: the affine case},
  journal={Internat. J. Math.},
  volume={24},
  date={2013},
  number={9},
  pages={1350069, 47},
  issn={0129-167X},
  review={\MR {3109441}},
  doi={10.1142/S0129167X13500699},
}

\bib{Popescu73}{book}{
  author={Popescu, N.},
  title={Abelian categories with applications to rings and modules},
  series={London Mathematical Society Monographs, No. 3},
  publisher={Academic Press, London-New York},
  date={1973},
  pages={xii+467},
  review={\MR {0340375}},
}

\bib{Shah23}{article}{
  author={Shah, Amit},
  title={Krull-Remak-Schmidt decompositions in Hom-finite additive categories},
  journal={Expo. Math.},
  volume={41},
  date={2023},
  number={1},
  pages={220--237},
  issn={0723-0869},
  review={\MR {4557280}},
  doi={10.1016/j.exmath.2022.12.003},
}

\bib{Simson05}{article}{
  author={Simson, Daniel},
  title={On Corner type Endo-Wild algebras},
  journal={J. Pure Appl. Algebra},
  volume={202},
  date={2005},
  number={1-3},
  pages={118--132},
  issn={0022-4049},
  review={\MR {2163404}},
  doi={10.1016/j.jpaa.2005.01.012},
}

\bib{SimsonS07v2}{book}{
  author={Simson, Daniel},
  author={Skowro\'{n}ski, Andrzej},
  title={Elements of the representation theory of associative algebras. Vol. 2},
  series={London Mathematical Society Student Texts},
  volume={71},
  note={Tubes and concealed algebras of Euclidean type},
  publisher={Cambridge University Press, Cambridge},
  date={2007},
  pages={xii+308},
  isbn={978-0-521-54420-7},
  review={\MR {2360503}},
}

\bib{SimsonS07v3}{book}{
  author={Simson, Daniel},
  author={Skowro\'{n}ski, Andrzej},
  title={Elements of the representation theory of associative algebras. Vol. 3},
  series={London Mathematical Society Student Texts},
  volume={72},
  note={Representation-infinite tilted algebras},
  publisher={Cambridge University Press, Cambridge},
  date={2007},
  pages={xii+456},
  isbn={978-0-521-70876-0},
  review={\MR {2382332}},
}

\bib{stacks-project}{misc}{
  author={The Stacks Project Authors},
  title={Stacks Project},
  url={http://stacks.math.columbia.edu},
  eprint={http://stacks.math.columbia.edu},
  date={2018},
}

\bib{StanleyW11}{article}{
  author={Stanley, Donald},
  author={Wang, Binbin},
  title={Classifying subcategories of finitely generated modules over a Noetherian ring},
  journal={J. Pure Appl. Algebra},
  volume={215},
  date={2011},
  number={11},
  pages={2684--2693},
  issn={0022-4049},
  review={\MR {2802159}},
  doi={10.1016/j.jpaa.2011.03.013},
}

\bib{Vipond20}{article}{
  author={Vipond, Oliver},
  title={Multiparameter persistence landscapes},
  journal={J. Mach. Learn. Res.},
  volume={21},
  date={2020},
  pages={Paper No. 61, 38},
  issn={1532-4435},
  review={\MR {4095340}},
}

\bib{Zhang19}{book}{
  author={Zhang, Yihui},
  title={Decomposition of Certain Representations Into A Direct Sum of Indecomposable Representations},
  note={Thesis (M.Sc.)--University of Regina},
  date={2019-04},
  pages={(no paging)},
  url={http://hdl.handle.net/10294/9033},
}

\end{biblist}
\end{bibdiv}
\end{document}